\documentclass[11pt]{amsart}

\usepackage{amsthm}

\usepackage{geometry}          
\usepackage{graphicx}
\usepackage[all]{xy}
\usepackage{subfig}
\usepackage{amssymb}
\usepackage{epstopdf}
\usepackage{emp}
\usepackage{color}
\usepackage{hyperref}

\usepackage{tikz}
\usetikzlibrary{matrix}
\usepackage[all]{xy}

\include{braids.mp}
\theoremstyle{definition}
\newtheorem{definition}{Definition}

\newtheorem*{definition*}{Definition}

\theoremstyle{plain}

\newtheorem*{proposition*}{Proposition}
\newtheorem{proposition}{Proposition}
\newtheorem*{corollary*}{Corollary}

\newtheorem{lemma}[proposition]{Lemma}
\newtheorem*{lemma*}{Lemma}

\newtheorem*{theorem*}{Theorem}

\newcommand{\R}{\mathbb{R}}  
\newcommand{\Z}{\mathbb{Z}}  

\title[Properties of the singular virtual braid monoid]{Algebraic, combinatorial and topological properties of singular virtual braid monoids}

\author{Bruno Aar\'on CISNEROS DE LA CRUZ}
\address{Bruno Aar\'on Cisneros de la Cruz, Instituto de Matem\'aticas de la UNAM, Oaxaca - Consejo Nacional de Humanidades Ciencias y Tecnolog\'ias \\ Le\'on No.2 altos \\ Oaxaca de Ju\'arez, M\'exico}
\email{bruno@im.unam.mx}

\author{Guillaume GANDOLFI}
\address{Guillaume Gandolfi, Laboratoire de Math\'{e}matiques Nicolas Oresme - Universit\'e de Caen Normandie, Caen, France}
\email{guillaume.gandolfi@unicaen.fr}

\newcommand{\dg}{\mathrm{deg}}

\begin{document}

\maketitle

\begin{abstract}
In this paper we discuss algebraic, combinatorial and topological properties of singular virtual braids. On the algebraic side we state the relations between classical and virtual singular objects, in addition we discuss a Birman--like conjecture for the virtual case. On the topological and combinatorial side, we prove that there is a bijection between singular abstract braids, horizontal Gauss diagrams and singular virtual braids, in particular using horizontal Gauss diagrams we obtain a presentation of the singular pure virtual braid monoid. 
\end{abstract}

\section{Introduction}

Recently, Caprau, De la Pena and McGahan defined singular virtual braids \cite{CPM} as a generalization of classical singular braids defined by Birman and Baez for the study of Vassiliev invariants \cite{B, Bae}, and virtual braids defined by Vershinin and Kauffman \cite{Ver01, Kauffman}. In \cite{CPM} they proved an Alexander and Markov Theorem for singular virtual braids and gave two presentations for $SVB_n$. Later, Caprau and Zepeda \cite{CZ} constructed a representation of $SVB_n$, and using Reidemeister--Schreier method they found a presentation for the Pure singular virtual braid monoid. In this paper we study the algebraic, combinatorial and topological context of singular virtual braids.

First we show that the singular virtual braid monoid is a natural extension of the singular braid monoid 
and of the virtual braid group
by proving that they are algebraically embedded in it (Proposition \ref{prop:FirstRelations}).

We state a Birman--like conjecture for singular virtual braids; Birman originally defined a map from the singular braid monoid to the braid group algebra called desingularization map, and conjectured that it was injective. This conjecture is known as ``Birman's conjecture". Independently Papadima and Bar-Natan \cite{PAPADIMA2002,BARNATAN} proved that Vassiliev invariants separated braids and it was proved by Zhu in \cite{Z} that if Birman's conjecture was true, then they would also separate singular braids. Later, Birman's conjecture was proved by Paris in \cite{LP}. 
We define the desingularization map for singular virtual braids and we prove that the preimage of $1$ by it is reduced to $\{1\}$ (Proposition \ref{prop}), which leads us to conjecture that it is injective.

We study singular virtual braids from the combinatorial point of view; Goussarov, Polyak and Viro defined Gauss diagrams as a combinatorial approach to study and calculate Vassiliev invariants \cite{GPV}. They discovered that virtual knots are in bijective correspondence with Gauss diagrams, identified up to $\Omega$--moves, which are roughly speaking the counterpart of Reidemeister moves in the context of Gauss diagrams.
We extend the definition of horizontal Gauss diagrams\footnote{In \cite{BACDLC}, they are called braid--Gauss diagrams.} \cite{ABMW,BARNATAN,BACDLC}, which are a braid--like version of Gauss diagrams, identified by their $\Omega$--moves.
We prove that these are in a bijective correspondence with singular virtual braids (Proposition \ref{prop:GaussD}) and as an application we recover the presentation of the pure singular virtual braid monoid given in \cite{CZ} without using Reidemeister--Schreier method (Proposition \ref{prop:PresPSVB}). 

We give a topological interpretation of singular virtual braids; for virtual knots it was given independently by Kauffman \cite{Kauffman} and by Kamada \cite{Kamada-tachi} as knot diagrams on surfaces (up to stable equivalence), called abstract knots, which extends the classical topological knot theory. The first author defined abstract braids and proved that they are a topological interpretation for virtual braids which is compatible with the objects defined by Kauffman and Kamada \cite{BACDLC}.
We extend this notion and define singular abstract braids, and we prove that they are in a bijective correspondence with singular virtual braids (Proposition \ref{abst_bd}).



The paper is organized as follows, in section \ref{Section2} we recall the presentations of the monoids and groups that we use in the paper. Section \ref{Section3} is dedicated to the algebraic context of singular virtual braids i.e. we prove that the virtual braid group and the singular braid monoid embed in the singular virtual braid monoid, and we state a Birman--like conjecture for the virtual case.  In section \ref{Section4} we give a combinatorial description of singular virtual braids as horizontal Gauss diagram and we recover the presentation of the pure singular virtual braid monoid. Finally, in section \ref{Section5} we use horizontal Gauss diagrams to establish a topological realization of singular virtual braids.

\section*{Acknowledgments}

The authors thank Paolo Bellingeri for bringing this subject to them, as well we thank Benjamin Audoux and Luis Paris for the fruitful discussions on the redaction and contents of this paper.

The first author was financed by CONAHCYT-M\'exico under the program of ``C\'atedras CONAHCYT, Proyecto 61", the research project ``Ciencia b\'asica 284621" and ``FORDECYT 265667". The second author's PhD is financed by the region of Normandie, France.

\section{Basic definitions}\label{Section2}

\subsection{Definitions}

\

We start by recalling the presentations of the different monoids and groups considered in this article. Set $n\geq 2$ a natural number.

\begin{definition}
The {\it  braid group on $n$ strands}, $B_n$, is the abstract group generated by $\sigma_1, \dots, \sigma_{n-1}$ with the following relations: 
\begin{itemize}
    \item[(R0)] $\sigma_i\sigma_j=\sigma_j\sigma_i,\ |i-j|\geq 2$,
	\item[(R3)] $\sigma_i\sigma_{i+1}\sigma_i=\sigma_{i+1}\sigma_i\sigma_{i+1},\ i=1,...,n-2$.
\end{itemize}
The {\it virtual braid group on $n$ strands}, $VB_n$, is the abstract group generated by $\sigma_1, \dots, \sigma_{n-1}$ (classical generators), $\rho_1, \dots, \rho_{n-1}$ (virtual generators), relations (R0), (R3), and
\begin{itemize}
	\item[(V1)] $\rho_i\rho_j=\rho_j\rho_i,\ |i-j|\geq 2$,
	\item[(V2)] $\sigma_i\rho_j=\rho_j\sigma_i,\ |i-j|\geq 2$,
	\item[(V3)] $\rho_i^2=1,\ i=1,...,n-1$,
	\item[(V4)] $\rho_i\rho_{i+1}\rho_i=\rho_{i+1}\rho_i\rho_{i+1},\ i=1,...,n-2$,
    	
	\item[(V5)] $\rho_i\sigma_{i+1}\rho_i=\rho_{i+1}\sigma_i\rho_{i+1},\ i=1,...,n-2$;
\end{itemize}
\end{definition}

Let us recall that there is a topological interpretation of braids as isotopy classes of strands embedded in the unital cube of $\R^3$, such that they move monotonically with respect to the $x$-axis, joining $n$--marked points on opposite faces, and that an isotopy is a continuous path of boundary--fixing diffeomorphisms of the unital cube, starting at the identity and preserving monotony of the strands.

Braids can be represented also through diagrams in the unital square of $\R^2$, such that they move monotonically with respect to the $x$-axis, joining $n$--marked points on opposite faces, and identified up to isotopy, which corresponds to relation (R0), and Reidemeister moves, which correspond to relations\footnote{The Reidemeister move (R2) corresponds to the inversibility of the generators $\sigma_i$, which is implicit in the group presentation.} (R2) and (R3) (see Figure \ref{fig:R2R3}). The generator $\sigma_i$ corresponds to the diagram represented on the left of Figure \ref{fig:Gen}, the generator $\sigma_i^{-1}$ is obtained from $\sigma_i$ by making a cross changing.

\begin{figure}[!ht]
  $\sigma_i$ \,\,\, =\,\,\,  \raisebox{-17pt}{\includegraphics[angle=90,height=0.7in,width=0.7in]{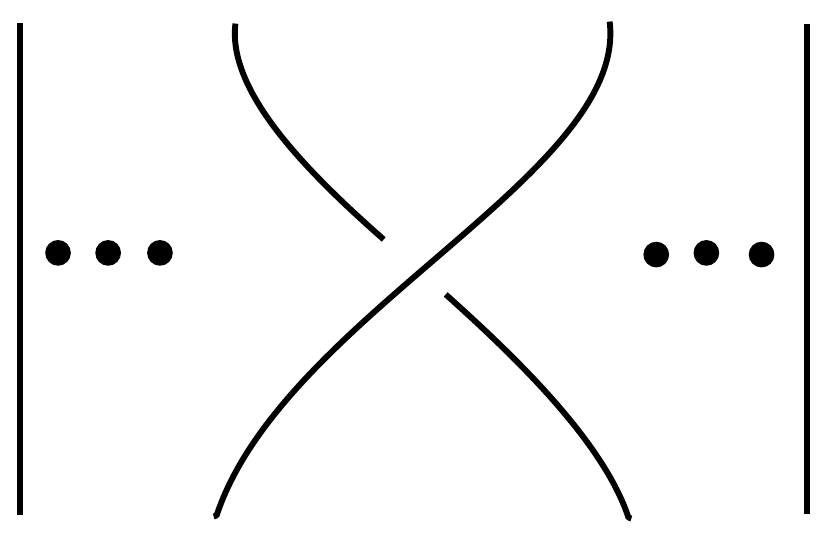}} 
  \hspace{1cm} 
  $\rho_i$ \,\,\,=\,\,\, \raisebox{-17pt}{\includegraphics[angle=90,height=0.7in,width=0.7in]{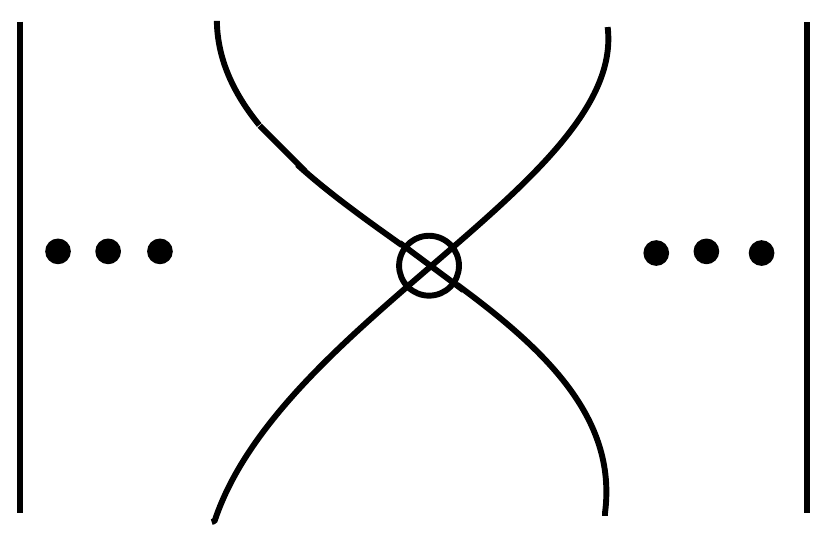}}
   \hspace{1cm} 
  $\tau_i$ \,\,\,=\,\,\, \raisebox{-17pt}{\includegraphics[angle=90,height=0.7in,width=0.7in]{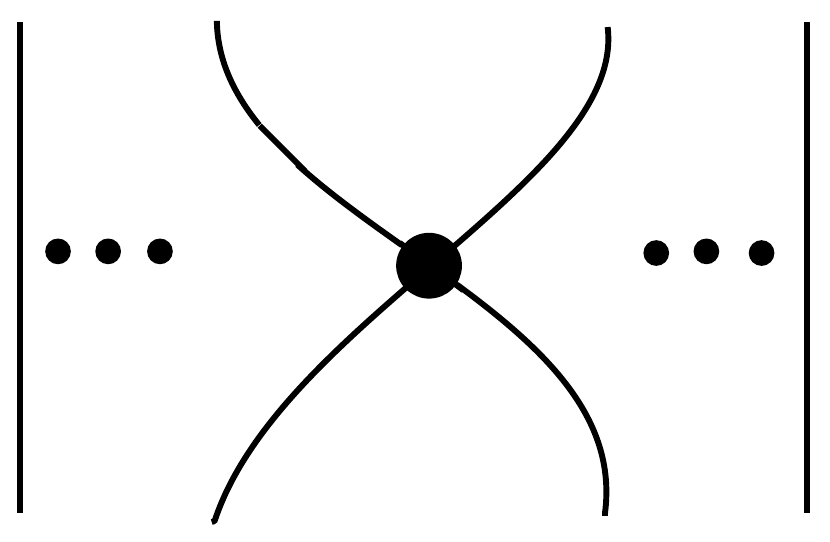}}
\put(-243, 30){\fontsize{7}{7}$1$}
\put(-243, 20){\fontsize{7}{7}$i$}
\put(-243, -6){\fontsize{7}{7}$i+1$}
\put(-243, -17){\fontsize{7}{7}$n$}
\put(-122, 30){\fontsize{7}{7}$1$}
\put(-122, 20){\fontsize{7}{7}$i$}
\put(-122, -6){\fontsize{7}{7}$i+1$}
\put(-122, -17){\fontsize{7}{7}$n$}
\put(0, 30){\fontsize{7}{7}$1$}
\put(0, 20){\fontsize{7}{7}$i$}
\put(0, -6){\fontsize{7}{7}$i+1$}
\put(0, -17){\fontsize{7}{7}$n$}
    \caption{Classical, virtual and singular generators}
    \label{fig:Gen}
\end{figure}

\begin{figure}[!ht]
\raisebox{-0.7cm}{\includegraphics[angle=90,scale=1]{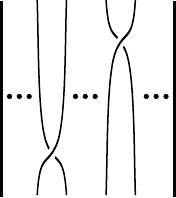}}
\hspace{0.2cm} $\stackrel{R0}{=}$\hspace{0.2cm}
\raisebox{-0.7cm}{\includegraphics[angle=90,scale=1]{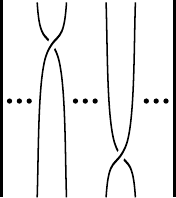}}
\vspace{1cm}\\
\raisebox{-1.6cm}{\includegraphics[angle=90,height=1in,width=0.8in]{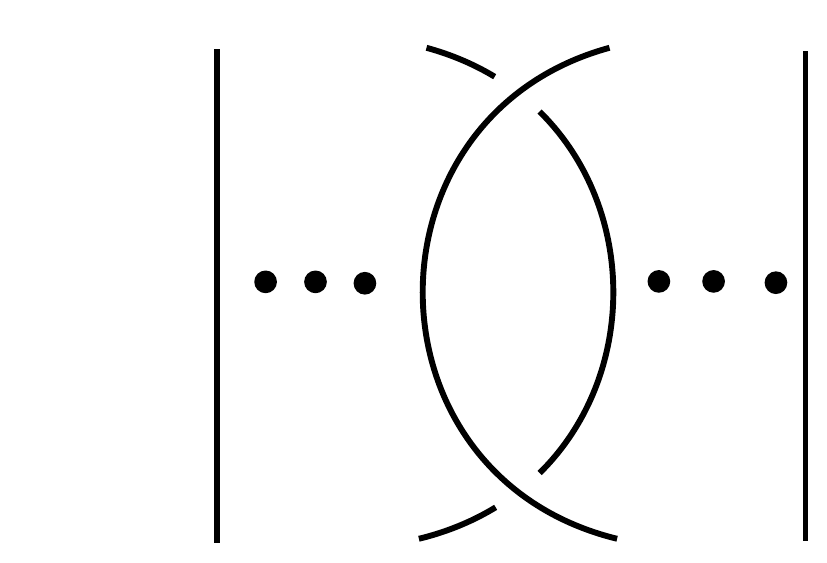}} 
\hspace{0.2cm} $\stackrel{R2}{=}$\hspace{0.2cm}
\raisebox{-1.05cm}{\includegraphics[angle=90,height=0.8in,width=0.9in]{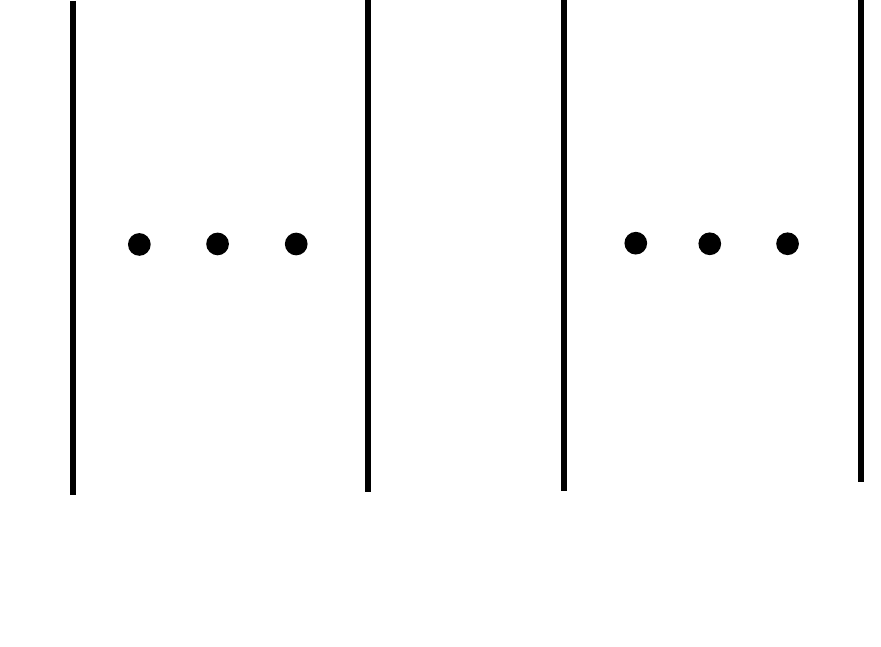}}
\hspace{1cm}
\raisebox{-.8cm}{\includegraphics[angle=90,height=0.7in,width=0.8in]{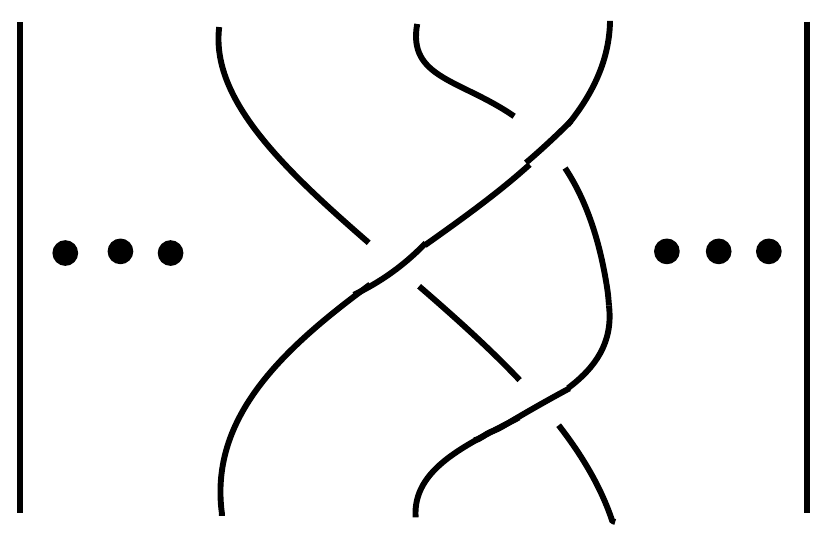}} \hspace{0.2cm} $\stackrel{R3}{=}$ \hspace{0.2cm} 
\raisebox{-.8cm}{\includegraphics[angle=90,height=0.7in,width=0.8in]{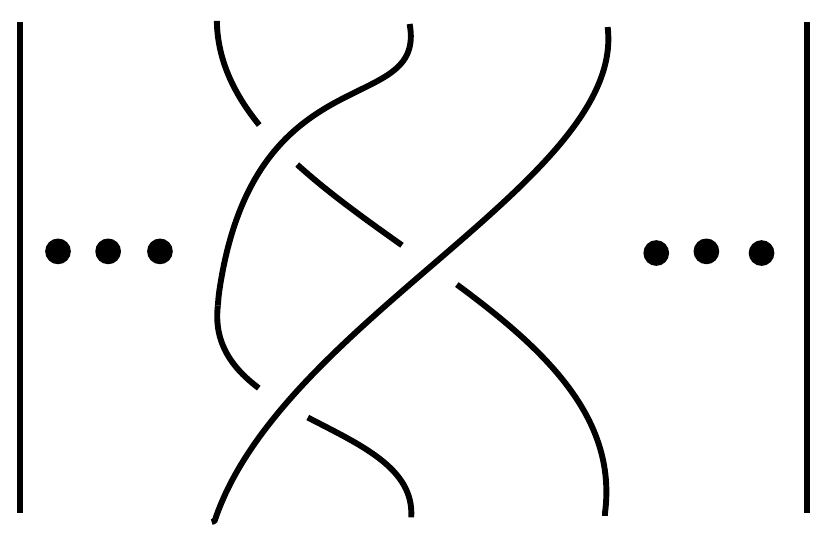}}

\caption{Relations R0, R2 and R3}
\label{fig:R2R3}
\end{figure}

Virtual braids were originally defined through braid-like diagrams, which are braid diagrams but with different labels on the crossings. As in the classical case, locally each crossing corresponds to a generator. Diagrammatically, for virtual braid diagrams, we have classical ($\sigma_i$) and virtual generators ($\rho_i$) illustrated in Figure \ref{fig:Gen}. Virtual braid diagrams are identified up to isotopy and moves (R0),(R2),(R3), isotopy (V1), (V2) and virtual Reidemeister moves (V3) to (V5) (see Figure \ref{fig:Vmoves}).

\begin{figure}[!ht]
\raisebox{-0.7cm}{\includegraphics[angle=90]{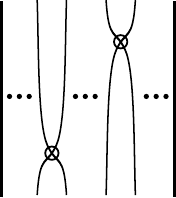}} 
\hspace{0.2cm} $\stackrel{V1}{=}$\hspace{0.2cm}
\raisebox{-0.7cm}{\includegraphics[angle=90]{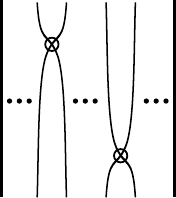}}
\hspace{1cm}
\raisebox{-0.7cm}{\includegraphics[angle=90]{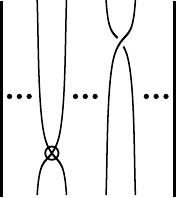}} 
\hspace{0.2cm} $\stackrel{V2}{=}$ \hspace{0.2cm} 
\raisebox{-0.7cm}{\includegraphics[angle=90]{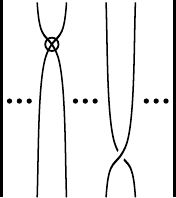}}
\vspace{1cm}\\
\raisebox{-1.6cm}{\includegraphics[angle=90, height=1in, width=0.8in]{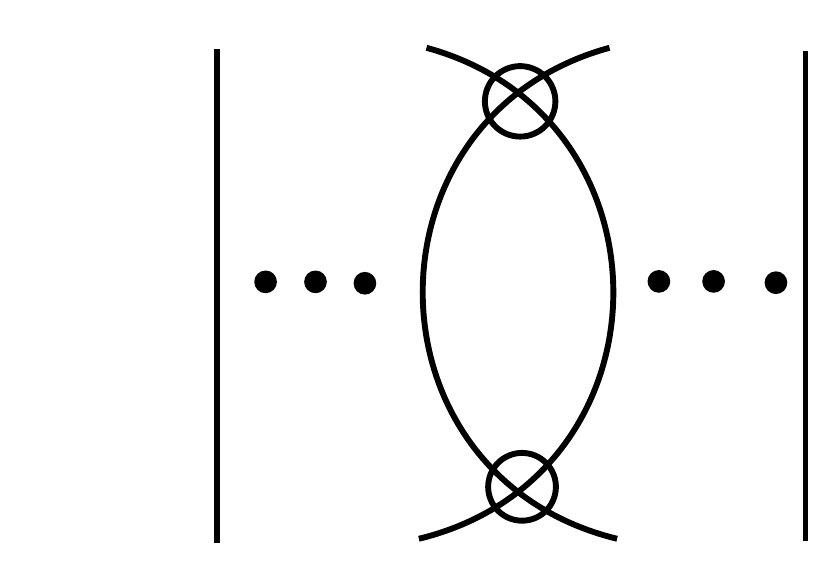}} 
\hspace{0.15cm} $\stackrel{V3}{=}$\hspace{0.25cm}
\raisebox{-1.05cm}{\includegraphics[angle=90, height=0.8in, width=0.8in]{Id_B.pdf}}
\hspace{0.2cm}
\raisebox{-0.8cm}{\includegraphics[angle=90, height=0.7in, width=0.7in]{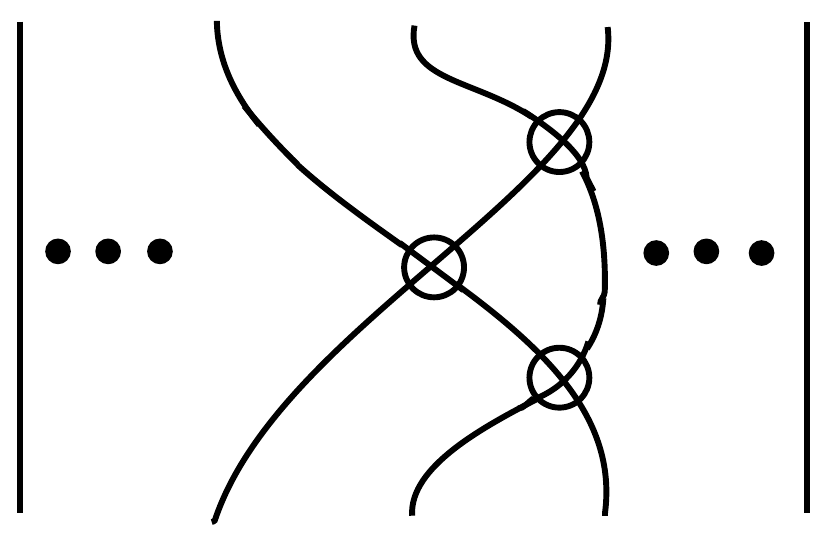}} \hspace{0.15cm} $\stackrel{V4}{=}$ \hspace{0.15cm} 
\raisebox{-.8cm}{\includegraphics[angle=90, height=0.7in, width=0.7in]{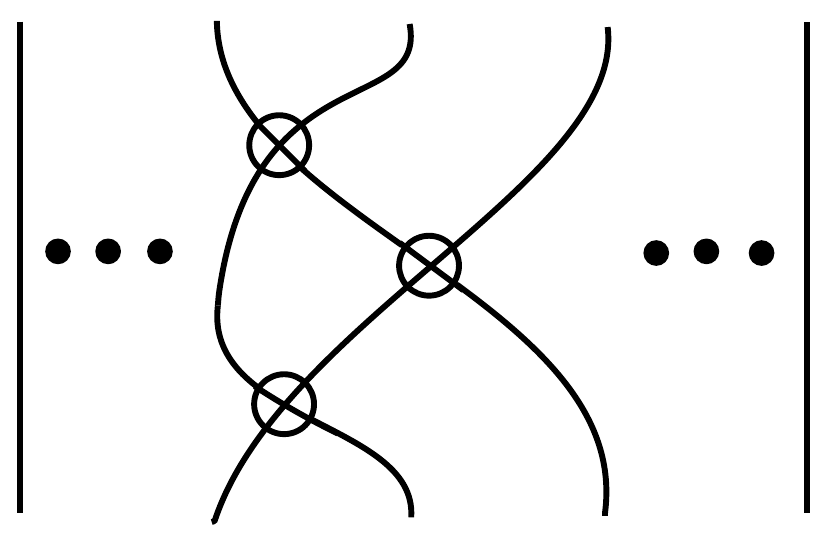}}
\hspace{0.7cm}
\raisebox{-.8cm}{\includegraphics[angle=90, height=0.7in, width=0.7in]{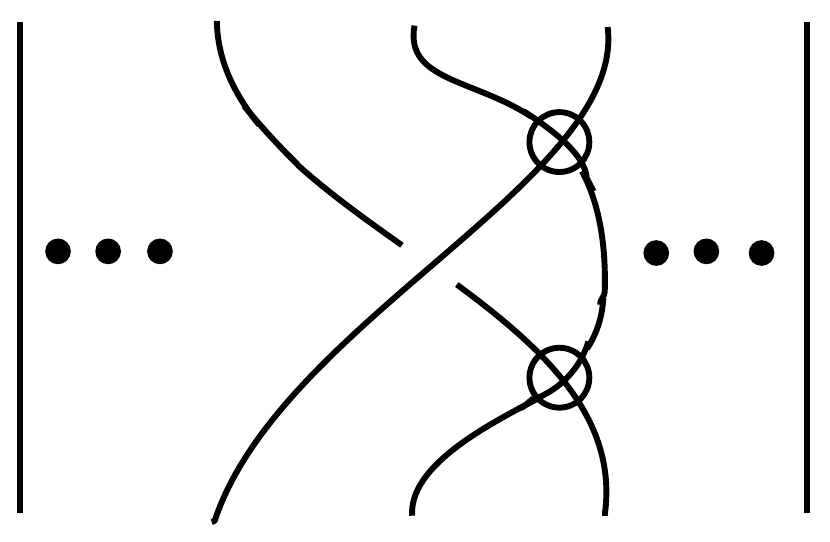}} \hspace{0.15cm} $\stackrel{V5}{=}$ \hspace{0.15cm} 
\raisebox{-.8cm}{\includegraphics[angle=90, height=0.7in, width=0.7in]{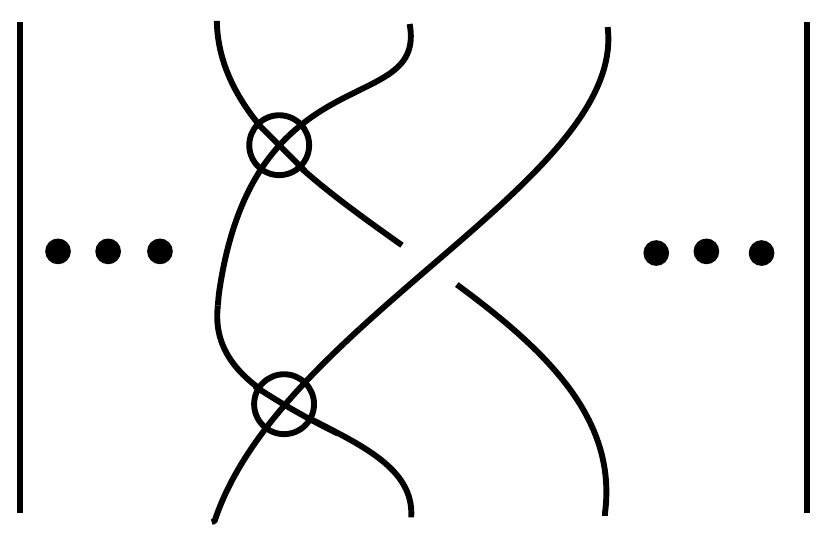}}
\caption{Relations V1 to V5}
\label{fig:Vmoves}
\end{figure}


\begin{definition}
The {\it  singular braid monoid on $n$ strands}, $SB_n$ is the abstract monoid generated by $\sigma_1^{\pm 1}, \dots, \sigma_{n-1}^{\pm 1}$ (classical generators), and $\tau_1, \dots, \tau_{n-1}$ (singular generators), with relations (R0), (R3) and 
\begin{itemize}
    \item[(R2)] $\sigma_i\sigma_i^{-1}=1 = \sigma_i^{-1}\sigma_i$,
    \item[(S1)] $\tau_i\tau_j=\tau_j\tau_i,\ |i-j|\geq 2$,
    \item[(S2)] $\tau_i \sigma_j = \sigma_j \tau_i,\ |i-j|\geq 2$,
    \item[(S3)] $\tau_i\sigma_i=\sigma_i\tau_i,\ i=1,...,n-1$,
    \item[(S4)] $\sigma_i\sigma_{i+1}\tau_i=\tau_{i+1}\sigma_i\sigma_{i+1},\ i=1,...,n-2$;
\end{itemize}
The {\it  singular virtual braid monoid on $n$ strands}, $SVB_n$, is the abstract monoid generated by $\sigma_1^{\pm 1}, \dots, \sigma_{n-1}^{\pm 1}$ (classical generators), $\rho_1, \dots, \rho_{n-1}$ (virtual generators), and $\tau_1, \dots, \tau_{n-1}$ (singular generators), with relations (R0), (R2), (R3), (S1) to (S4), (V1) to (V5) and  
\begin{itemize}
    \item[(SV1)] $\rho_i\tau_j=\tau_j\rho_i,\ |i-j|\geq 2$,
    \item[(SV2)] $\rho_i\tau_{i+1}\rho_i=\rho_{i+1}\tau_i\rho_{i+1},\ i=1,...,n-2$.
\end{itemize}
\end{definition}

Singular braid monoids have a topological counterpart as isotopy classes of strands immersed in the unital cube of $\R^3$, such that they move monotonically respect to the $x$-axis, joining $n$--marked points on opposite faces and their singularities are isolated transverse double points. Notice that, since diffeomorphisms preserve singular points, so does an isotopy.
As in the case of classical braids they admit a diagrammatic representation, then we have two types of crossings, classical ($\sigma_i$) and singular ($\tau_i$), which correspond to the generators illustrated in Figure \ref{fig:Gen}. The diagrams of singular braids are identified up to isotopy and moves (R0), (R2), (R3), isotopy (S1), (S2), move (S3) and the singular Reidemeister move (S4) (see Figure \ref{fig:Smoves}).  

\begin{figure}[!ht]
\hspace{0.2cm}
\raisebox{-0.7cm}{\includegraphics[angle=90]{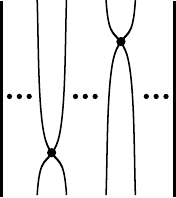}} 
\hspace{0.1cm} $\stackrel{S1}{=}$\hspace{0.1cm}
\raisebox{-0.7cm}{\includegraphics[angle=90]{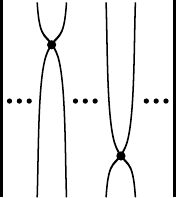}}
\hspace{0.3cm}
\raisebox{-0.7cm}{\includegraphics[angle=90]{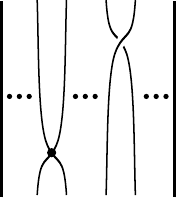}} 
\hspace{0.1cm} $\stackrel{S2}{=}$ \hspace{0.1cm} 
\raisebox{-0.7cm}{\includegraphics[angle=90]{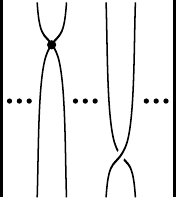}}
\vspace{1cm}\\
\raisebox{-1.3cm}{\includegraphics[angle=90,height=1in,width=0.8in]{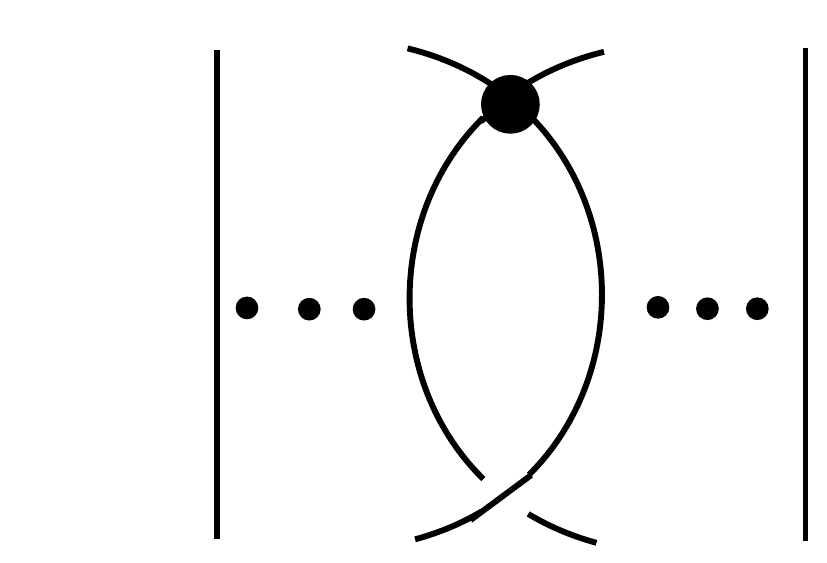}} 
\hspace{0.1cm} $\stackrel{S3}{=}$ \hspace{0.1cm} 
\raisebox{-1.3cm}{\includegraphics[angle=90,,height=1in,width=0.8in]{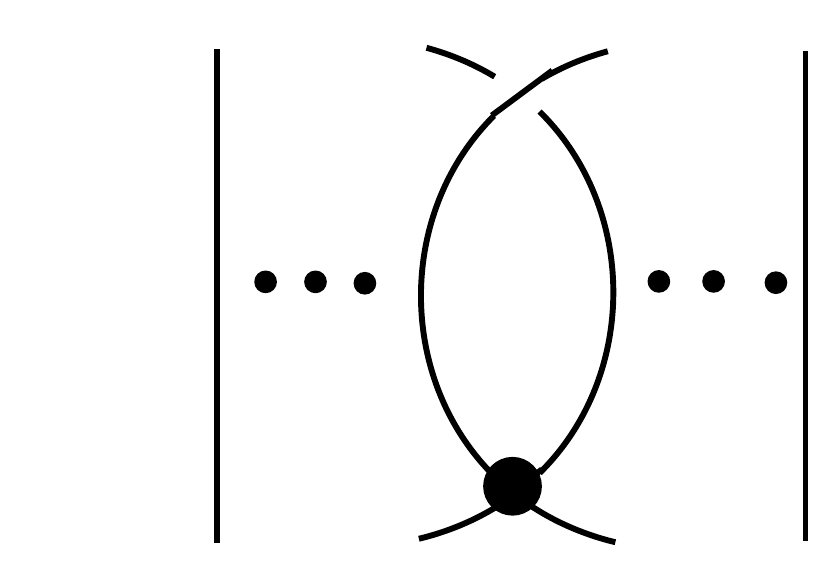}}
\hspace{0.2cm}
\raisebox{-.7cm}{\includegraphics[angle=90, height=0.75in, width=0.8in]{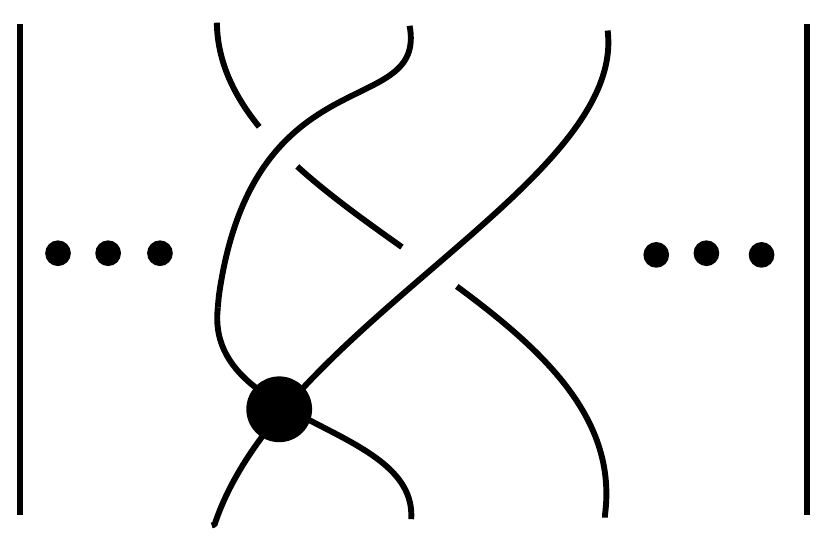}} 
\hspace{0.2cm} $\stackrel{S4}{=}$ \hspace{0.2cm} 
\raisebox{-.7cm}{\includegraphics[angle=90, height=0.75in, width=0.8in]{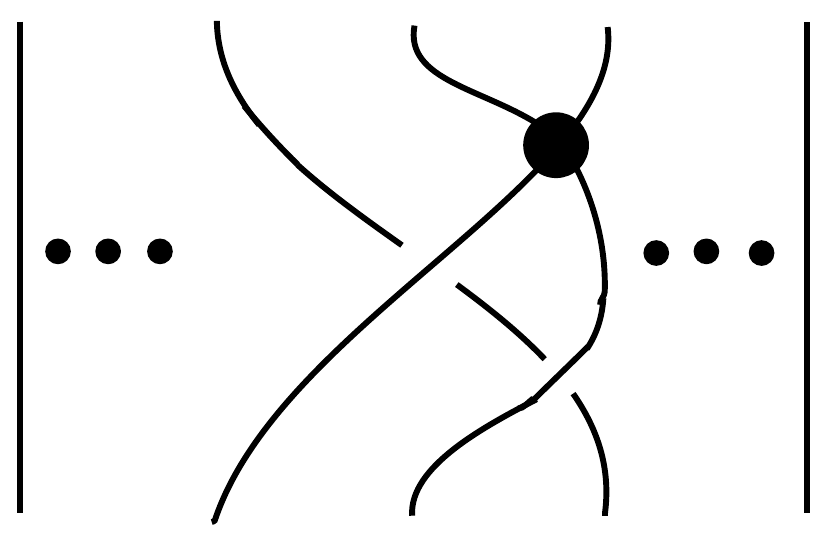}}
\caption{Relations S1 to S4}
\label{fig:Smoves}
\end{figure}

Similarly to the non-singular case, the singular virtual braid monoid is defined through braid-like diagrams, there are three types of crossings, classical ($\sigma_i$), virtual ($\rho_i$) and singular ($\tau_i$), which correspond to the generators illustrated in Figure \ref{fig:Gen}. The diagrams of virtual singular braids are identified up to isotopy and moves (R0), (R2), (R3), (V1) to (V5), (S1) to (S4), isotopy (SV1) and the singular virtual Reidemeister move (SV2) (see Figure \ref{fig:SVmoves}).

\begin{figure}[!ht]
\raisebox{-.7cm}{\includegraphics[angle=90]{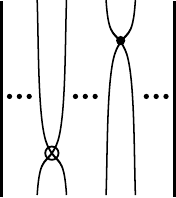}} 
\hspace{0.2cm} $\stackrel{SV1}{=}$ \hspace{0.2cm} 
\raisebox{-.7cm}{\includegraphics[angle=90]{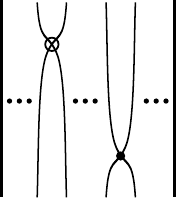}}
\hspace{1cm}
\raisebox{-.7cm}{\includegraphics[angle=90, height=0.7in, width=0.7in]{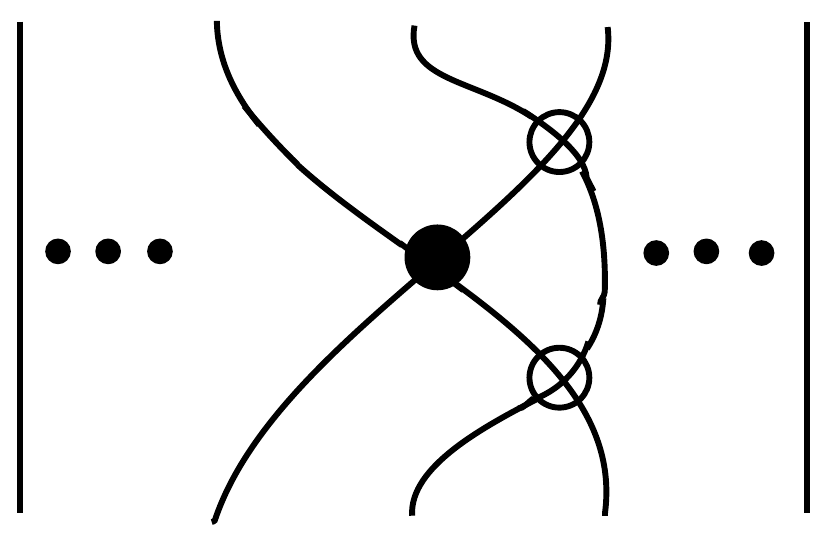}} \hspace{0.2cm} $\stackrel{SV2}{=}$ \hspace{0.2cm} 
\raisebox{-.7cm}{\includegraphics[angle=90, height=0.7in, width=0.7in]{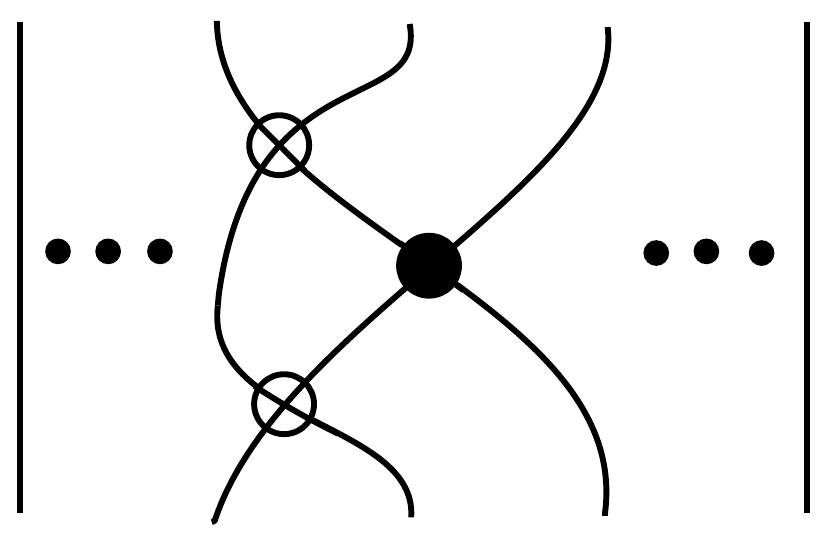}}
\caption{Relations SV1 and SV2}
\label{fig:SVmoves}
\end{figure}

\section{Algebraic properties}\label{Section3}

In this section we discuss the relation between classical, virtual and singular braid objects and we present some evidence to state a Birman--like conjecture for the virtual case. 

\subsection{Relations between the objects}

Consider the following diagram, where each arrow represents the natural morphism between the objects.

\begin{displaymath}
\xymatrix{
VB_n\ar[r] & SVB_n \\
B_n \ar[u] \ar[r] & SB_n \ar[u] .}
\end{displaymath}

\begin{proposition} \label{prop:FirstRelations}
The previous diagram is commutative and each morphism is injective. 
\end{proposition}

It is easy to check the commutativity of the diagram. Furthermore, the injectivity of the map from $B_n$ to $VB_n$ was already proved in \cite{Kamada}. We are going to prove the injectivity of the other maps. 
In order to do this, we first need to show that $SB_n$ and $SVB_n$ can be decomposed into certain semidirect products.

Define the two monoids: $$\mathcal{M} := \langle  \; \Upsilon \; | uv = vu, \; u,v\in \Upsilon \;  \text{ and } uv=vu \ \text{in}\ SB_n \rangle^+$$ where $\Upsilon=\{\beta\tau_i\beta^{-1},\ \beta\in B_n,\ 1\leq i\leq n-1\}$,  and $$\mathcal{M}_v := \langle  \; \Upsilon_v \; | uv = vu, \; u,v\in \Upsilon_v  \;  \text{ and } uv=vu \ \text{in}\ SVB_n \;  \rangle^+ $$ where $\Upsilon_v=\{\beta\tau_i\beta^{-1},\ \beta\in VB_n,\ 1\leq i\leq n-1\}$.\\
Notice that $B_n$ (resp. $VB_n$) acts on $\mathcal{M}$ (resp. $\mathcal{M}_v$) by conjugation. 

\begin{lemma} The singular (virtual) braid monoid admit the following decompositions
$$SB_n=\mathcal{M}\rtimes B_n \text{     and     }SVB_n=\mathcal{M}_v\rtimes VB_n.$$
\end{lemma}

\begin{proof}
We prove that the decomposition for singular virtual braids holds, following \cite{LP} where a similar decomposition is made for the classical singular braid monoid. 

On the free monoid over $\sigma_i^{\pm 1},\rho_i,\tau_i$, for $1\leq i \leq n-1 $, define the homomorphism to $\mathcal{M}_v\rtimes VB_n$ which sends every $\sigma_i, \rho_i$ to $(1,\sigma_i),(1,\rho_i)$, respectively, and every $\tau_i$ to $(\tau_i,1)$. It is easy to verify that the images of the generators satisfy the defining relations of $SVB_n$, thus the homomorphism can be factorized through $SVB_n$.
Call this map $\Phi$. 

On the other hand, define the map from $M_v \rtimes VB_n$ to $SVB_n$ as follows, $(\omega, \beta)$ goes to $\omega \beta$, it is clear that this is a monoid homomorphism, we call it $\Psi$. Notice that $\Psi \circ \Phi = Id_{SVB_n}$ and $\Phi \circ \Psi = Id_{\mathcal{M}_v\rtimes VB_n}$. From this we have the desired decomposition.

\end{proof}

Now we can finish the proof of Proposition \ref{prop:FirstRelations}.

\begin{proof}
\
The injection of $B_n$ (resp $VB_n$) in $SB_n$ (resp. in $SVB_n$) follows from these decompositions. It remains to prove the injection of $SB_n$ in $SVB_n$, which is reduced to the problem of the injectivity of the natural maps $B_n\rightarrow VB_n$ and $\kappa: \mathcal{M}\rightarrow \mathcal{M}_v$. The injectivity of the first map has already been mentioned, so we only need to prove the injectivity of $\kappa$.

The core of the proof is to show that if $\beta_0,\beta_1\in B_n$ and $1\leq i,j\leq n-1$ then the following are equivalent\footnote{Actually, only $(1)\Rightarrow (2)$ is needed but the converse is obvious.}:
\begin{itemize}
\item[(1)] $(\beta_0\tau_i\beta_0^{-1})(\beta_1\tau_j\beta_1^{-1})=(\beta_1\tau_j\beta_1^{-1})(\beta_0\tau_i\beta_0^{-1})$ holds in $SVB_n$
\item[(2)] $(\beta_0\tau_i\beta_0^{-1})(\beta_1\tau_j\beta_1^{-1})=(\beta_1\tau_j\beta_1^{-1})(\beta_0\tau_i\beta_0^{-1})$ holds in $SB_n$.
\end{itemize}

Let $\hat{\theta} : SVB_n\rightarrow VB_n$ be the map defined as $\hat{\theta} (\sigma_i) = \sigma_i$, $\hat{\theta} (\rho_i) = \rho_i$ and $\hat{\theta} (\tau_i) = \sigma_i$. It is easy to verify that $\hat{\theta}$ is a well-defined monoid homomorphism using the defining relations of $SVB_n$ and $SB_n$. Now suppose that $(1)$ holds, then $\hat{\theta}((\beta_0\tau_i\beta_0^{-1})(\beta_1\tau_j\beta_1^{-1}))=\hat{\theta}((\beta_1\tau_j\beta_1^{-1})(\beta_0\tau_i\beta_0^{-1}))$ i.e.:
$$(\beta_0\sigma_i\beta_0^{-1})(\beta_1\sigma_j\beta_1^{-1})=(\beta_1\sigma_j\beta_1^{-1})(\beta_0\sigma_i\beta_0^{-1})\text{ in }VB_n.$$
Since $B_n$ embeds in $VB_n$, the equality holds in $B_n$ and therefore in $SB_n$. Theorem 7.1 from \cite{FRZ} states that for every singular braid $\beta$ and $1\leq i,j\leq n-1$ then $\beta\sigma_i=\sigma_j\beta$ if and only if $\beta\tau_i=\tau_j\beta$ and as a consequence we have (in $SB_n$):
$$\begin{array}{cc}
 & (\beta_0\sigma_i\beta_0^{-1})(\beta_1\sigma_j\beta_1^{-1})=(\beta_1\sigma_j\beta_1^{-1})(\beta_0\sigma_i\beta_0^{-1}) \\
\Leftrightarrow & \sigma_i(\beta_0^{-1}\beta_1\sigma_j\beta_1^{-1}\beta_0) = (\beta_0^{-1}\beta_1\sigma_j\beta_1^{-1}\beta_0)\sigma_i \\
\Leftrightarrow & \tau_i(\beta_0^{-1}\beta_1\sigma_j\beta_1^{-1}\beta_0) = (\beta_0^{-1}\beta_1\sigma_j\beta_1^{-1})\beta_0)\tau_i \\
\Leftrightarrow & (\beta_0\tau_i\beta_0^{-1})(\beta_1\sigma_j\beta_1^{-1})=(\beta_1\sigma_j\beta_1^{-1})(\beta_0\tau_i\beta_0^{-1}) \\
\Leftrightarrow & (\beta_1^{-1}\beta_0\tau_i\beta_0^{-1}\beta_1)\sigma_j=\sigma_j(\beta_1^{-1}\beta_0\tau_i\beta_0^{-1}\beta_1) \\
\Leftrightarrow & (\beta_1^{-1}\beta_0\tau_i\beta_0^{-1}\beta_1)\tau_j=\tau_j(\beta_1^{-1}\beta_0\tau_i\beta_0^{-1}\beta_1) \\
\Leftrightarrow & (\beta_0\tau_i\beta_0^{-1})(\beta_1\tau_j\beta_1^{-1})=(\beta_1\tau_j\beta_1^{-1})(\beta_0\tau_i\beta_0^{-1}).
\end{array}$$
Hence the result.
\end{proof}

\subsection{Birman--like conjecture}
\

Recall that if $M$ is a monoid then $\mathbb{Z}[M]$ is the $\mathbb{Z}$-algebra whose underlying module is the free $\mathbb{Z}$-module over the elements $M$ and which is endowed with the multiplication obtained by extending bilinearly the multiplication of $M$. Recall also that $\mathbb{Z}$ can be identified with $\mathbb{Z}.e$ where $e$ denotes the identity element of $M$. 

Let $\eta: SB_n \rightarrow \Z[B_n]$ be the {\it desingularization map} defined by  $\eta(\tau_i) = \sigma_i - \sigma_i^{-1}$ and $\eta(\sigma_i^{\pm 1})= \sigma_i^{\pm 1}$. This map was defined by Birman and Baez \cite{B, Bae} in order to study finite type invariants. Birman conjectured that this map was injective and this was proved by Paris in \cite{LP}.  We can define a similar map in the virtual case as follows. Let $\hat{\eta}: SVB_n \rightarrow \Z[VB_n]$ defined by $\hat{\eta}(\tau_i) = \sigma_i - \sigma_i^{-1}$, $\hat{\eta}(\sigma_i^{\pm 1}) = \sigma_i^{\pm 1}$ and $\hat{\eta}(\rho_i) = \rho_i$. We obtain the following commutative diagram. 

\begin{displaymath}
\xymatrix{
SVB_n\ar[r]^{\hat{\eta}} & \Z[VB_n] \\ 
SB_n \ar @{^{(}->}[u] \ar@{^{(}->}[r]^{\eta} & Z[B_n] \ar@{^{(}->}[u] }
\end{displaymath}

Let us recall that for monoid homomorphisms, contrary to group homomorphisms, having a trivial kernel is a necessary but not sufficient condition to be injective. Therefore the following proposition gives only a partial answer to the question of the injectivity of $\eta$. It is strongly inspired from \cite{FRZ}, where the classical case is treated. 


\begin{proposition}\label{prop}
\
Let $a\in \mathbb{Z}$, then:
$$\hat{\eta}^{-1}(\{a\})=
\begin{cases}
1\ \text{if}\ a=1,\\
\emptyset \ \text{otherwise.}
\end{cases}$$
\end{proposition}

To prove this proposition, we define a degree on elements of $SVB_n$.

\begin{definition}
Let $\dg$ be the monoid homomorphism from $SVB_n$ onto $(\mathbb{Z},+)$ that maps every $\sigma_i^{\pm 1}$ to $\pm 1$ and the other generators to 0. The image by $\dg$ of an element in $SVB_n$ is called the {\it degree} of this element.
\end{definition}

Notice that for every defining relation for $SVB_n$ the number of singularity in the right-hand side of the equation is the same as in the left-hand side. As a consequence, two words on the generators of $SVB_n$ representing the same element always have the same number of singularities. The number of singularities of a singular virtual braid is therefore well-defined, as well as the following definition :

\begin{definition}
For every $d\in\mathbb{N}$, we denote by $S_dVB_n$ the subset of $SVB_n$ constituted of braids with exactly $d$ singularities.
\end{definition}

The following shows an example of the decomposition of $\eta(\beta)$ in $\Z[VB_n]$ of a braid $\beta\in VB_n$.

Let $\omega = \rho_1 \sigma_2^{-1} \tau_1 \rho_2\sigma_2 \tau_2\in S_2VB_3$, then 
\[\begin{split}
\hat{\eta}(\omega) & = \rho_1 \sigma_2^{-1} \sigma_1 \rho_2\sigma_2 \sigma_2 \\
                & \quad - \rho_1 \sigma_2^{-1} \sigma_1 \rho_2\sigma_2 \sigma_2^{-1}  \\
                & \quad -\rho_1 \sigma_2^{-1} \sigma_1^{-1} \rho_2\sigma_2 \sigma_2 \\
                & \quad + \rho_1 \sigma_2^{-1} \sigma_1^{-1} \rho_2\sigma_2\sigma_2^{-1}.
 \end{split}\]
 
Notice that $\hat{\eta}(\omega)$ is a sum of $2^2$ elements, that  $\mathrm{deg}(\omega)= 0$, that the degree of $\rho_1 \sigma_2^{-1} \sigma_1^{-1} \rho_2\sigma_2\sigma_2^{-1}$ is $0-2$, that the degree of $\rho_1 \sigma_2^{-1} \sigma_1 \rho_2\sigma_2 \sigma_2$ is $0+2$ and that those are the unique elements of maximal and minimal degree in the sum obtained by $\hat{\eta}$. 
This is stated for the general case in the next lemma.

\begin{lemma}
Let $\beta\in S_dVB_n$, with $\text{deg}(\beta)= s$, then:
\begin{itemize}
\item[1)] $\hat{\eta}(\beta)=\sum\limits_{i=1}^{2^d}a_i\alpha_i$ where $a_i\in\mathbb{Z}$ and $\alpha_i\in VB_n$,
\item[2)]there exist unique $k,l$ such that $\dg(\alpha_k)=s-d$ and $\dg(\alpha_l)=s+d$,
\item[3)]for every $i\ne\ k,l$, we have $s-d < \dg(\alpha_i) < s+d$.
\end{itemize}
\end{lemma}

\begin{proof}
We prove the proposition by induction on $d$:

If $\beta\in S_0VB_n=VB_n$ then $\hat{\eta}(\beta)=\beta$ and the result is trivial.

Assume the result holds for every $\beta\in S_dVB_n$ and then take $\beta\in S_{d+1}VB_n$. Let 
$\beta=x_1...x_m$ where the $x_i$ are generators of $SVB_n$ and let $j$ be such that $x_j=\tau_i$ for some $i\in \{1,\dots, n-1\}$ and $x_k\notin \{\tau_1,...,\tau_{n-1}\}$ for $k>j$. Then $\beta'=x_1...x_{j-1}$ has exactly $d$ singularities and therefore, by induction hypothesis, $\hat{\eta}(\beta')$ can be written as $a'_1\alpha'_1+...+a'_{2^d}\alpha'_{2^d}$ where the elements $\alpha'_{min},\ \alpha'_{max}$ of minimum and maximum degree have respectively degree $s'-d$ and $s'+d$ with $s'=\dg(\beta')=\dg(\beta)-\dg(x_j...x_m)$ and every other element has degree strictly between $s'-d$ and $s'+d$.
It follows that:
$$\hat{\eta}(\beta'x_j)=\hat{\eta}(\beta'\tau_i)=\sum\limits_{i=1}^{2^d}a'_i\alpha'_i\sigma_i-\sum\limits_{i=1}^{2^d}a'_i\alpha'_i\sigma_i^{-1},$$
and that $\varphi(\beta'x_j)$ is a sum of $2^{d+1}$ terms, each of which having a degree strictly between $s'-d-1$ and $s'+d+1$ except for $\alpha'_{min}\sigma_i^{-1},\ \alpha'_{max}\sigma_i$ which have respectively degree $s'-d-1$ and $s'+d+1$.
Finally, we get that $\hat{\eta}(\beta)=\hat{\eta}(\beta'x_j)x_{j+1}...x_m$ is a sum of $2^{d+1}$ terms each of which having a degree strictly between
$$s'-(d+1)+\dg(x_{j+1}...x_m)=s-(d+1)$$
and
$$s'+(d+1)+\dg(x_{j+1}...x_m)=s+(d+1)$$
 except for $\alpha'_{min}\sigma_i^{-1}x_{j+1}...x_m$ and $ \alpha'_{max}\sigma_i x_{j+1}...x_m$ which have respectively degree $s-(d+1)$ and $s+(d+1)$.
\end{proof}

Now we can prove Proposition \ref{prop}.

\begin{proof}
Let $a\in\mathbb{Z}$ and let $\beta\in SVB_n$ such that $\hat{\eta}(\beta)=1$. If $s$ and $d$ denote respectively the number of singularities and the degree of $\beta$, then by the previous proposition we have $s+d=s-d=0$, which implies that $d=0$, that is $\beta$ has no singularity and therefore that $\beta =\hat{\eta}(\beta)=a$. Since $SVB_n\cap\mathbb{Z}=\{1\}$, if $a\ne 1$ we end up with a contradiction and if $a=1$ then $\beta=1$. Hence the result.
\end{proof}


This results motivates the following Birman--like conjecture in the virtual case. \\

\noindent
{\bf Conjecture.} {\it  The desingularization map $\hat{\eta}: SVB_n \rightarrow \Z[VB_n]$ is injective.} \\

Finally, consider the linear extension $H$ of $\eta$ to $\Z[SB_n]$; this map was proved not to be injective in \cite{FRZ}. It follows that the extension $\hat{H}$ of $\hat{\eta}$ to $\Z[SVB_n]$ is not injective since it coincides with $H$ on $\Z [B_n]$, as shown in the following diagram :
\begin{displaymath}
\xymatrix{
\Z [SVB_n]  \ar[r]^{\hat{H}} & \Z [VB_n] \\
\Z [SB_n]  \ar@{^{(}->}[u] \ar[r]^{H} & \Z [B_n].\ar@{^{(}->}[u]}
\end{displaymath}

\section{Combinatorial properties of $SVB_n$} \label{Section4}

\subsection{Singular Gauss diagrams}

\

In \cite{BACDLC}, it is proved that the set of virtual braids are in a bijective correspondence with the set of stable equivalence classes of abstract braids, the bijection given there passes through horizontal Gauss diagrams.  In this section we extend the definition of horizontal Gauss diagrams to the singular case, and we prove that they are in bijective correspondence with singular virtual braids.  As an algebraic application of this bijection, we recover the presentation of the virtual pure singular braid monoid given by Caprau and Zepeda \cite{CZ}. 

\begin{definition}\label{def:GaussD}
A {\it  singular horizontal Gauss diagram on $n$ strands} $G$ is a 4--tuple $(I,A,S,\pi)$ where $I$ is a collection of $n$ oriented disjoint intervals embedded in the plane, such that they parallel to each other and ordered from top to bottom, they are called {\it underlying intervals}, $A$ is a finite set of signed arrows, $S$ is a finite set of unsigned arrows (we call unsigned arrows simply by arrows) and $\pi\in S_n$, where $S_n$ is the set of permutations of $\{1,\dots,n\}$, such that: 
\begin{enumerate}
	\item each (signed) arrow has its endpoints in the interior of two different underlying intervals,
	\item (signed) arrows are pairwise disjoint,
	\item the endpoint of the $i$-th underlying interval is labelled with $\pi(i)$. 
\end{enumerate}

\end{definition}

\noindent {\bf Remark. }Notice that condition (2) in Definition \ref{def:GaussD} implies that we can draw the arrows perpendicular to the underlying intervals, up to reparametrization of the underlying intervals. When we draw underlying intervals vertically, then arrows are horizontal. This is why they are called horizontal Gauss diagrams. In this paper we draw underlying intervals horizontally.

\begin{definition}
Let $D_1$ and $D_2$ be two singular horizontal Gauss diagrams. We say that $D_1$ and $D_2$ are related by an {\it  $\Omega$--move} if $D_1$ has a subdiagram equivalent to one of the diagrams shown on right (respectively left) side of pictures (A), (B), (C) and (D) of Figure \ref{fig:omegamoves}, and replacing this subdiagram for the one shown on the left (respectively right) side of the picture gives $D_2$.
Observe that in pictures (A), (B), (C) and (D), $i,j,k$ may not be consecutive, nor in increasing order.
Each move is labelled according to the subdiagrams that we change, i.e. we have four $\Omega$--moves: $\Omega 2$, $\Omega 3$, $S \Omega 2$ and $S \Omega 3$--moves\footnote{Notice that as in the case of (virtual) braids, we can have an $\Omega 3$--like move with different signs on the arrows by composing different $\Omega$--moves.}.
\end{definition}

\begin{figure}[!ht]
\subfloat[$\Omega 2$--move]{\includegraphics[scale=0.5]{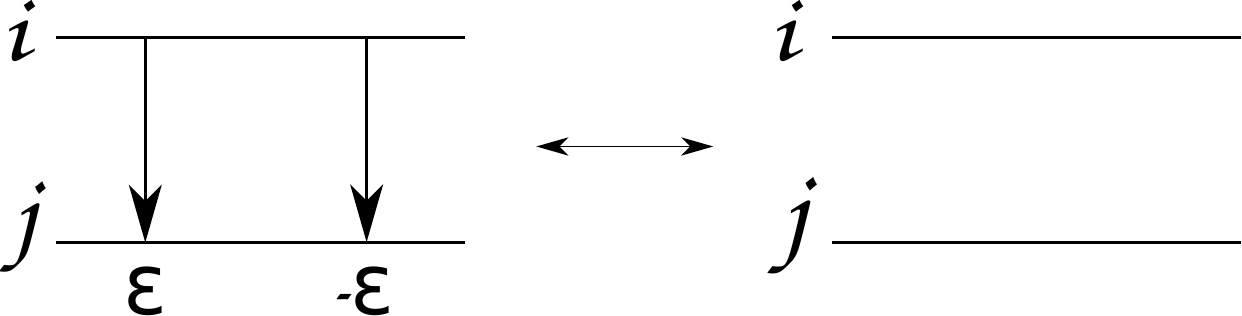}}
\hspace{40pt}
\subfloat[$\Omega 3$--move]{\includegraphics[scale=0.5]{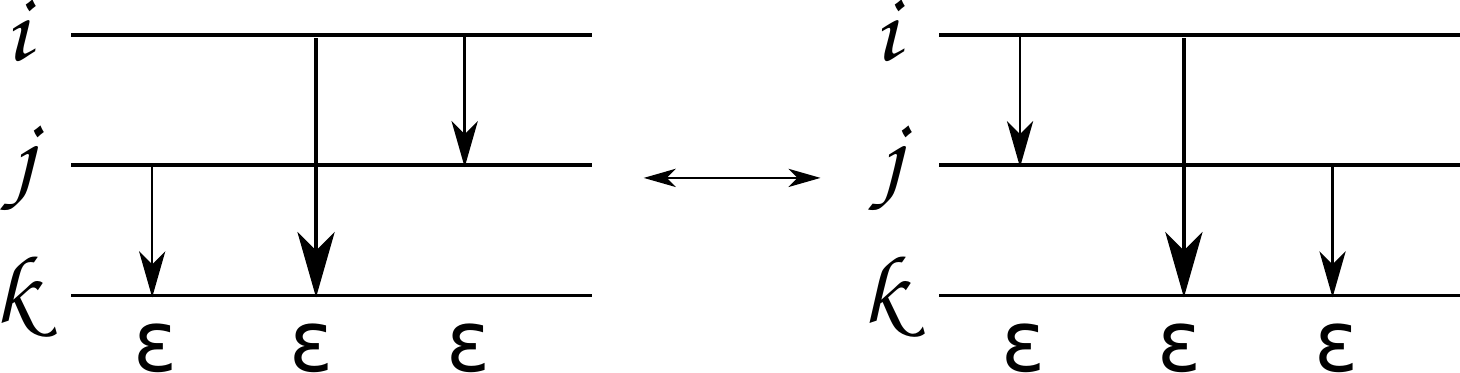}}
\vspace{15pt}\\
\subfloat[S$\Omega 2$--move]{\includegraphics[scale=0.5]{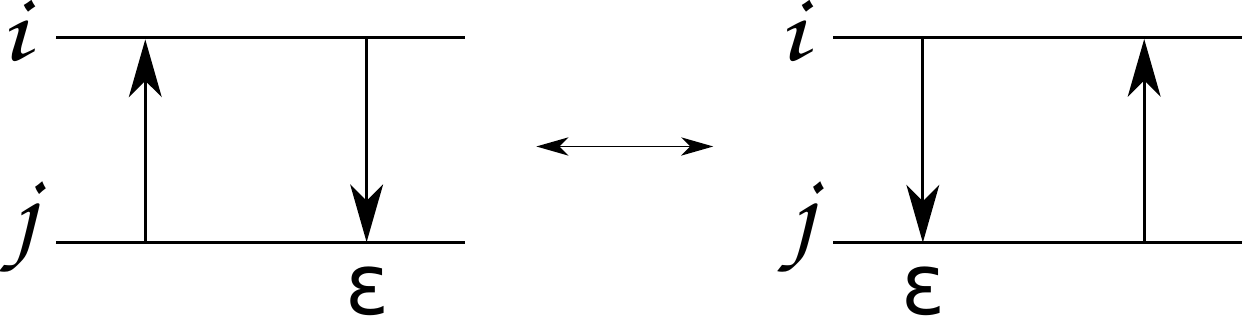}}
\hspace{40pt}
\subfloat[S$\Omega 3$--move]{\includegraphics[scale=0.5]{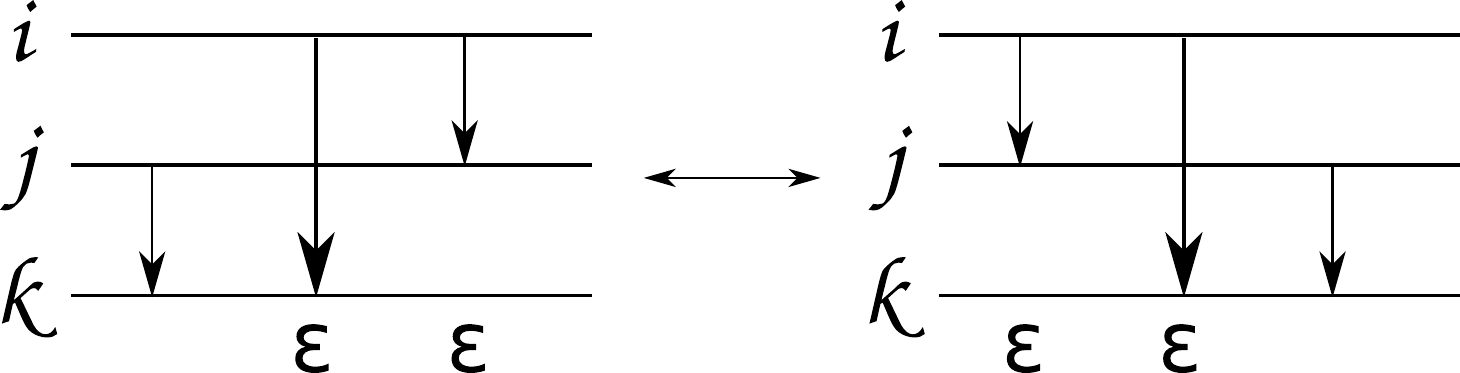}}
\caption{$\Omega$--moves}
\label{fig:omegamoves}
\end{figure}

Singular horizontal Gauss diagrams are identified by the equivalence relation generated by the $\Omega$--moves and oriented diffeomorphisms of the underlying intervals $I_i$, $i=1,\dots ,n$. We call {\it  horizontal Gauss diagrams} an equivalence class of singular horizontal Gauss diagrams and we denote the set of equivalence classes of horizontal Gauss diagrams on $n$ strands by $G_n$.  

\begin{definition}\label{def:GaussBraidDiagram}
Let $\beta$ be a virtual singular braid diagram on $n$ strands. The {\it  singular horizontal Gauss diagram of $\beta$}, $G(\beta)$, is a singular horizontal Gauss diagram on $n$ strands given by: 
\begin{itemize}
	\item each underlying interval of $G(\beta)$ is associated to the corresponding preimage of a strand of $\beta$, 
	\item there is a {\it  signed arrow} for each classical crossing, whose endpoints correspond to the preimages of the crossing with the following rule: 
	\begin{itemize}
		\item Arrows are pointing from the over-passing string to the under-passing string,
		\item The sign of the arrow is given by the sign of the crossing,
	\end{itemize}
	\item there is a {\it  simple arrow} for each singular crossing, whose endpoints correspond to the preimages of the singular crossing with the following rule: according to the standard orientation of the plane, the tail corresponds to the strand that plays the role of the $x$-axis, and the head of the arrow corresponds to the strand that plays the role of the $y$-axis:
\begin{figure}[ht]
\hspace{25pt}
\subfloat{\includegraphics[scale=0.35]{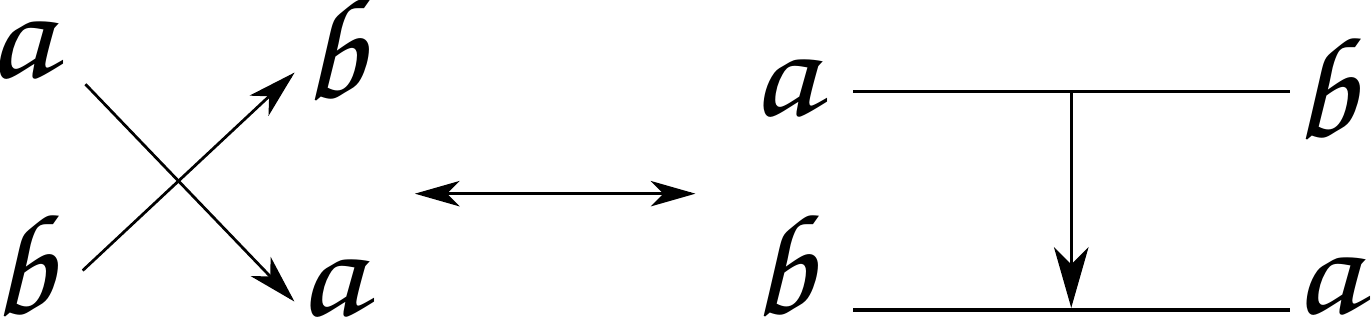}}
\hspace{40pt}
\subfloat{\includegraphics[scale=0.35]{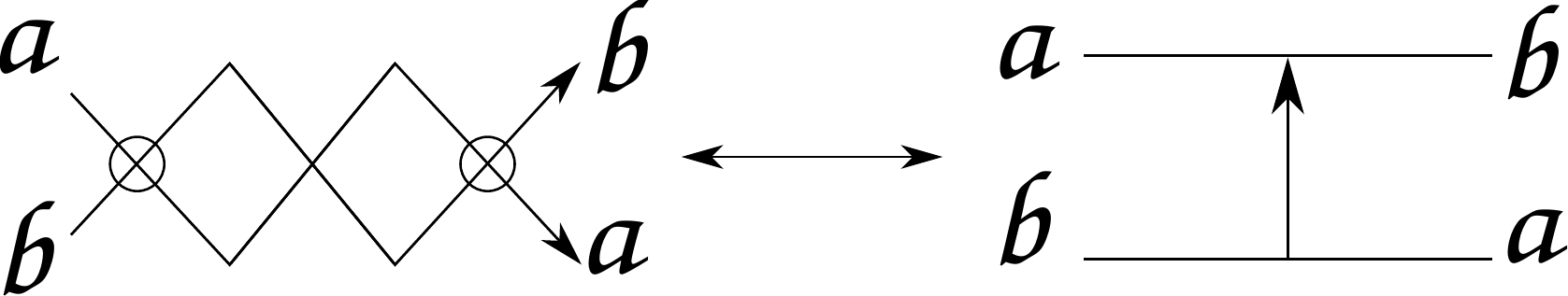}}
\end{figure}
	\item the permutation of $G(\beta)$ corresponds to the permutation associated to $\beta$.
\end{itemize}
\end{definition}

In Figure \ref{fig:HorGauss} is an example of translation of a singular virtual braid diagram into a horizontal Gauss diagram.

\begin{figure}[ht]
\includegraphics[scale=0.7]{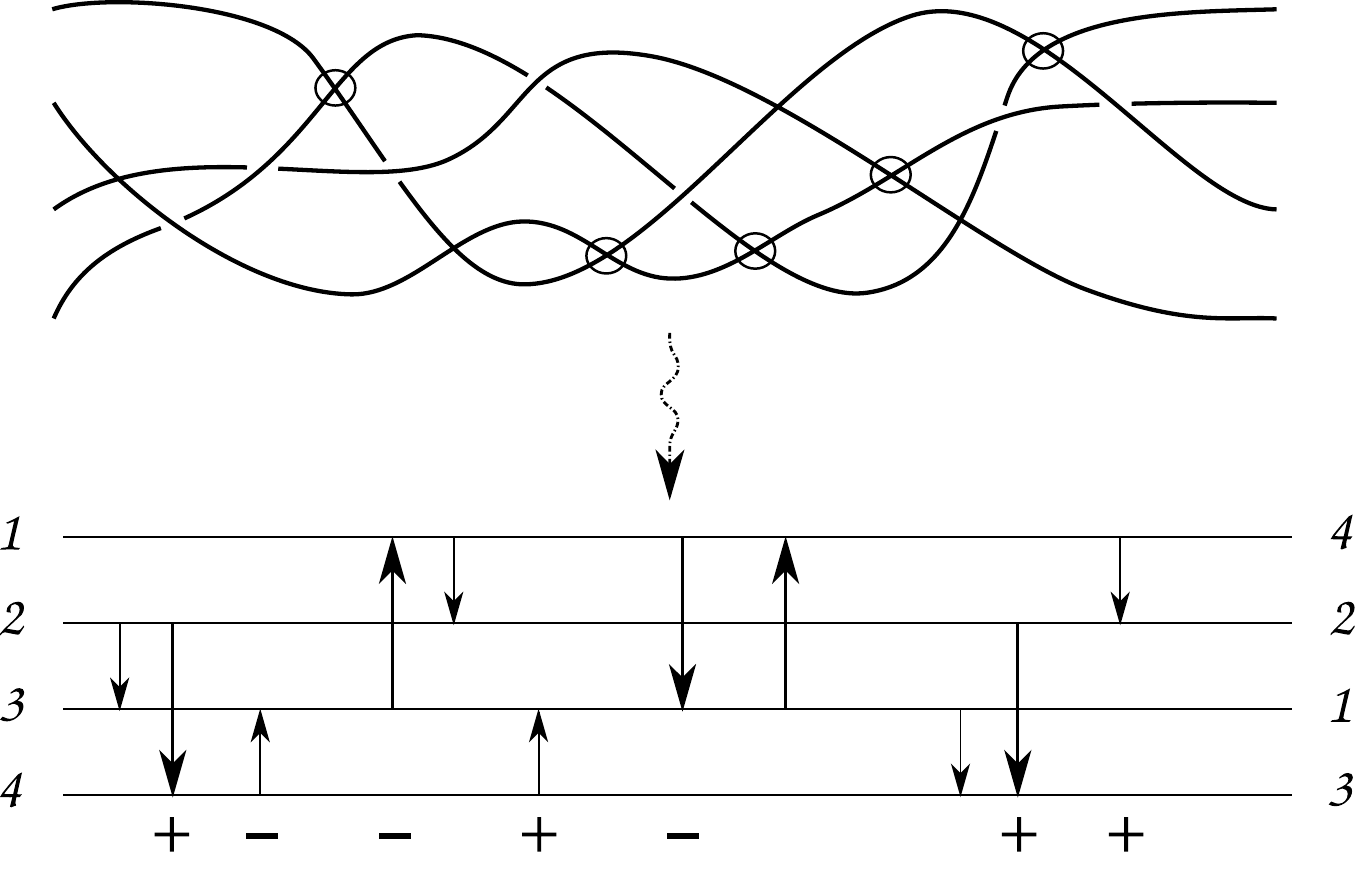}
\caption{The construction of $G(\beta)$}\label{fig:HorGauss}
\end{figure}

Notice that that classical, virtual and singular Reidemeister moves are translated into $\Omega$--moves, and isotopies are translated into diffeomorphisms of the underlying intervals. Thus, there is a well defined function from virtual singular braids to horizontal Gauss diagrams, leading to the next proposition.

\begin{proposition}\label{prop:GaussD}
There is a bijective correspondence between the set of singular virtual braids on $n$ strands, $SVB_n$, and the set of horizontal Gauss diagrams, $G_n$. 
\end{proposition}

\begin{proof}
We have a well defined function $G: SVB_n \rightarrow G_n$. It suffices to construct a function $B: G_n \rightarrow SVB_n$ and show that $G\circ B = Id_{G_n}$ and $B\circ G = Id_{SVB_n}$.

Given a singular horizontal Gauss diagram $g$, we can construct a singular virtual braid diagram, $B(g)$, as follows. 

Consider a unitary square on the plane, with $n$ points on the left, labelled by $1, \dots, n$, and $n$ points on the right, labelled by the permutation of the singular horizontal Gauss diagram. Then, draw a classical crossing, with the given sign for each signed arrow, and a singular crossing for each arrow, in the order that they appear on the diagram with respect to the $x$-axis and in such a way that their projection to the $x$-axis do not intersect. Notice that each underlying interval represents a strand of the virtual singular braid, indicating the crossing where the strand is involved, according to the (signed) arrows, and its endpoints with respect to the labelled points. Then the singular horizontal Gauss diagram can be seen as a set of instructions for drawing {\it  joining arcs} connecting crossings and endpoints. Notice that in the construction of the joining arcs, new crossings will appear. These crossings will be virtual. With this we have constructed a virtual singular braid diagram. 

We can verify that $B(g)$ is well defined up to  virtual Reidemeister moves, and if $g$ is identified up to oriented diffeomorphism, then $B(g)$ is well defined up to isotopy and virtual Reidemeister moves. 

This construction induces a well defined function from the set of equivalence classes of singular horizontal Gauss diagrams ($G_n$) to the set of equivalence classes of virtual singular braid diagrams ($SVB_n$) - for a detailed discussion about this see \cite{BACDLC}.  It is immediate to see that $B: G_n \rightarrow SVB_n$ and $G\circ B = Id_{G_n}$ and $B\circ G = Id_{SVB_n}$. For an example of a singular virtual braid built from a horizontal Gauss diagram, see Figure \ref{fig:BraidGauss}.
\end{proof}

\begin{figure}[ht]
\includegraphics[scale=0.7]{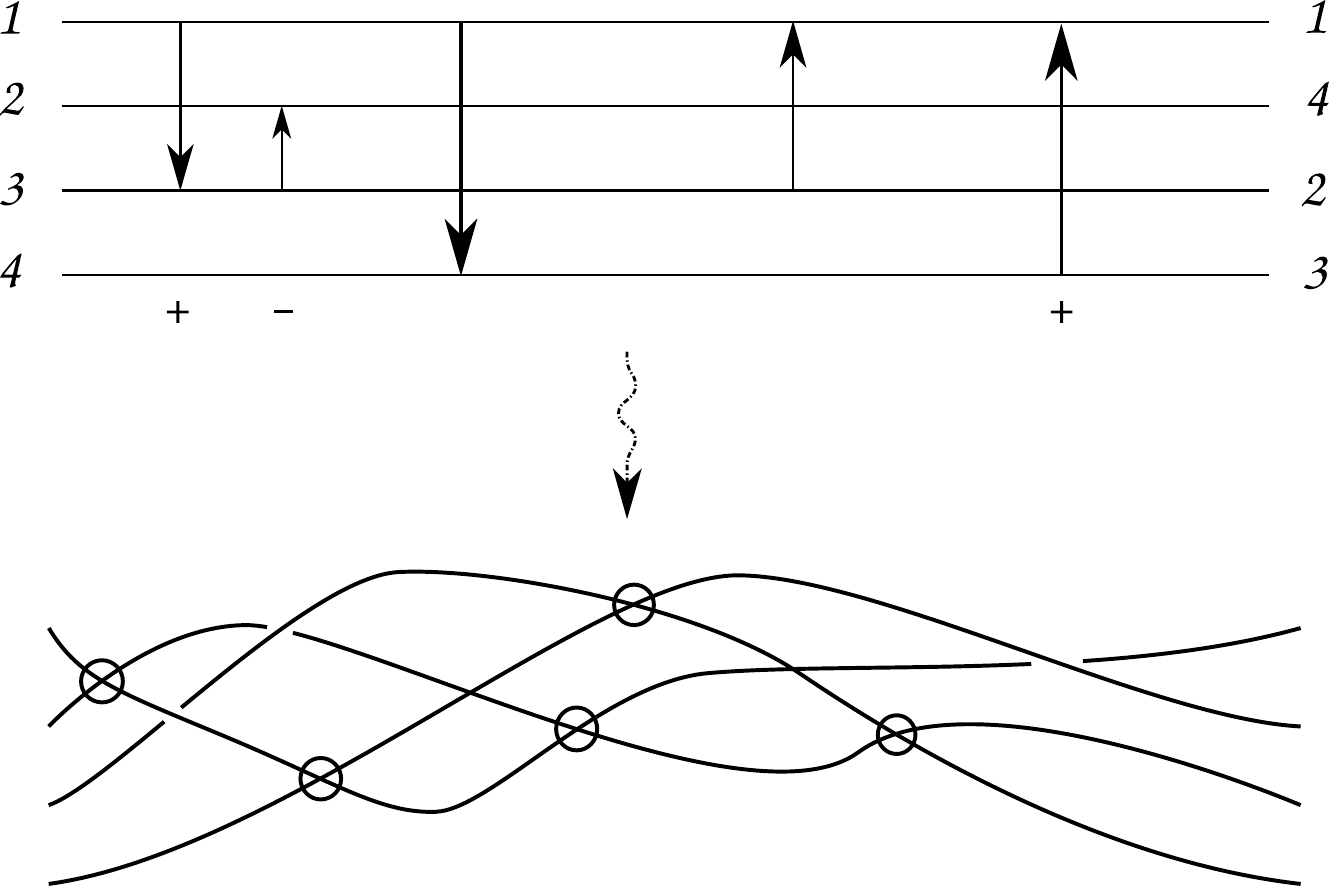} 
\caption{The construction of $B(g)$}\label{fig:BraidGauss}
\end{figure}

\subsection{Singular virtual pure braids on $n$ strands $SVP_n$}

\

Let $u_i\in \{\sigma_i, \tau_i, \rho_i \; | \; 1\leq i \leq n-1 \; \}$  be a generator of $SVB_n$, and let $\theta : SVB_n \rightarrow S_n$ be defined by $\theta(u_i) = (i,i+1)$. Notice that $\theta$ induces an homomorphism. The kernel of this homomorphism is called the monoid of {\it  singular virtual pure braids on $n$ strands}, denoted by $SVP_n$.  The elements of $SVP_n$ are singular virtual braid diagrams, identified up to classical, virtual and singular Reidemeister moves and isotopy, whose associated permutation is the identity (i.e. the strands do not mix the endpoints).  

As an application of the bijection between horizontal Gauss diagrams and singular virtual braids, we recover the presentation of $SVP_n$ given by Caprau and Zepeda \cite{CZ}.

\noindent
\begin{proposition}\label{prop:PresPSVB}
The virtual singular pure braid monoid on $n$ strands, $SVP_n$ admits the following presentation. 
\begin{itemize}
	\item Generators: For $\epsilon \in \{ \pm 1\}$ and $1\leq i \neq j \leq n$, $X^{\epsilon}_{i,j}$ and $Y_{i,j}$. 
	\item Relations: For $i,j,k,l\in \{1,\dots,n\}$ all different and $\epsilon \in \{\pm 1\}$,
	\begin{itemize}
		\item[(SP1)] $X_{i,j}^{\epsilon}X_{i,j}^{-\epsilon} = 1$. 
		\item[(SP2)] $X_{i,j}^{\epsilon}X_{i,k}^{\epsilon}X_{j,k}^{\epsilon} =  X_{j,k}^{\epsilon}X_{i,k}^{\epsilon}X_{i,j}^{\epsilon}$.
		\item[(SP3)] $X_{i,j} X_{k,l} = X_{k,l} X_{i,j}$, $Y_{i,j}Y_{k,l} = Y_{k,l} Y_{i,j}$ and $X_{i,j}Y_{k,l} = Y_{k,l} X_{i,j}$.
		\item[(SP4)] $Y_{i,j} X_{j,i}^{\epsilon} = X_{i,j}^{\epsilon} Y_{j,i}$. 
		\item[(SP5)] $Y_{j,k}X_{i,k}^{\epsilon}X_{i,j}^{\epsilon} =  X_{i,j}^{\epsilon}X_{i,k}^{\epsilon}Y_{j,k}$.
	\end{itemize}
\end{itemize}
\end{proposition}

\begin{proof}
 As proved in Proposition \ref{prop:GaussD}, there is a bijection between virtual braid diagrams and horizontal Gauss diagrams and therefore there is a bijection between the set of pure virtual braid diagrams and the set of horizontal Gauss diagrams whose associated permutation is the identity or equivalently horizontal Gauss diagrams ``without permutation".

Given a {\it  singular horizontal Gauss diagram} without permutation, we can express it in terms of its signed arrows and simple arrows as follows: 
\begin{itemize}
	\item up to reparametrization of the underlying intervals, we can suppose that for each time $t\in [0,1]$ there is exactly one arrow or signed arrow,
	\item denote a signed arrow by $A_{i,j}^{\epsilon}$ if it begins on the $i$-th interval and ends on the $j$-th interval, with $\epsilon \in \{\pm 1\}$ according to the sign of the signed arrow,
	\item denote a simple arrow by $S_{i,j}$ if it begins on the $i$-th interval and ends on the $j$-th interva,
	\item write a word concatenating $A_{i,j}^{\epsilon}$'s and $S_{i,j}$'s depending on the signed or simple arrows, as they appear on the singular horizontal Gauss diagram, according to the time. 
\end{itemize}
Conversely, given a word on the alphabet $\Sigma = \{ \ A_{i,j}^{\pm 1}, \ S_{i,j} \ | \ \text{for $1\leq i\neq j \leq n$ and $\epsilon \in \{\pm 1\}$} \}$, one can build a horizontal singular Gauss diagram without permutation.


From this discussion, there is a bijection between the set of singular horizontal Gauss diagrams without permutation (call it the set of {\it pure horizontal Gauss diagrams}) and the free monoid $\mathcal{F}(\Sigma)$ over $\Sigma$.

It remains to translate the $\Omega$--moves and diffeomorphisms on pure horizontal Gauss diagrams into relations on the monoid, to obtain the proposition. For $\epsilon \in\{\pm 1\}$ and $i,j,k,l \in \{1,\dots, n\}$ different: 

\begin{enumerate}
	\item[(P1)] move $\Omega2$ translates as $A_{i,j}^{\epsilon}A_{i,j}^{-\epsilon} = e$,
	\item[(P2)] move $\Omega3$ translates as $A_{i,j}^{\epsilon}A_{i,k}^{\epsilon}A_{j,k}^{\epsilon} =  A_{j,k}^{\epsilon}A_{i,k}^{\epsilon}A_{i,j}^{\epsilon}$,
	\item[(P3)] move $S\Omega2$ translates as $S_{i,j} A_{j,i}^{\epsilon} = A_{i,j}^{\epsilon} S_{j,i}$,
	\item[(P4)] move $S\Omega3$ translates as $S_{j,k}A_{i,k}^{\epsilon}A_{i,j}^{\epsilon} =  A_{i,j}^{\epsilon}A_{i,k}^{\epsilon}S_{j,k}$,
	\item[(P5)] reparametrization of the underlying intervals generates the  following relations: 
	\begin{itemize}
	    \item $A_{i,j} A_{k,l} = A_{k,l} A_{i,j}$,
	    \item $S_{i,j}S_{k,l} = S_{k,l} S_{i,j}$,
	    \item$A_{i,j}S_{k,l} = S_{k,l} A_{i,j}$.
	\end{itemize}
\end{enumerate}

Furthermore, as pure singular virtual braid diagrams (pure horizontal Gauss diagrams) do not mix the strands, multiplication of pure singular virtual braids (concatenation of pure horizontal Gauss diagrams) is equivalent to the product of the monoid $\mathcal{F}(\Sigma)$. As a consequence, the homomorphism given by $\varphi: \mathcal{F}(\Sigma) \rightarrow SVP_n$ defined for $1\leq i < j \leq n$ as follows: 
\[\begin{split}
	\varphi(A_{i,j}^{\epsilon})  	&= X_{i,j} := \rho_{j-1} \dots \rho_{i+1} \rho_i \sigma_i \rho_i  \rho_{i+1} \dots \rho_{j-1} \\
	\varphi(A_{j,i}^{\epsilon}) 	&= X_{j,i} :=\rho_{j-1} \dots \rho_{i+1}  \sigma_i   \rho_{i+1} \dots \rho_{j-1} \\
	\varphi(S_{i,j})			&= Y_{i,j} := \rho_{j-1} \dots \rho_{i+1} \rho_i \tau_i \rho_i  \rho_{i+1} \dots \rho_{j-1} \\
	\varphi(S_{j,i})			&= Y_{j,i} := \rho_{j-1} \dots \rho_{i+1}  \tau_i  \rho_{i+1} \dots \rho_{j-1} 
\end{split}\] 
is well defined and induces a monoid isomorphism between $\mathcal{F}(\Sigma)/(P1-P5)$ and $SVP_n$.
\end{proof}

\noindent

\noindent
{\bf Remark. } The homomorphism $\theta: SVB_n \rightarrow S_n$ has a monoid section, $\tau: S_n \rightarrow SVB_n$, defined on its generators by $\tau((i,i+1))= \rho_i$, for $1\leq i \leq n-1$. This gives a decomposition $$SVB_n=SVP_n \rtimes S_n$$ where the action of $S_n$ on $SVP_n$ is given by $\pi \cdot X_{i,j} = X_{\pi(i), \pi(j)}$ and $\pi \cdot Y_{i,j} = Y_{\pi(i), \pi(j)}$. \\

\section{Topological properties}\label{Section5}

In this section, we exhibit a topological realization of singular virtual braids as singular abstract braids, which generalizes abstract braids introduced in \cite{BACDLC} and \cite{Kamada-tachi}.

\begin{definition}
A {\it  singular abstract braid diagram on $n$ strands} is $\bar{\beta}= (S,\beta,\epsilon)$ is a triple such that: 
\begin{enumerate}
	\item $S$ is a connected, compact and oriented surface with $\partial S= C_0 \sqcup C_1$,
	\item each boundary component of $S$ has $n$ marked points, say $\{a_1, \dots, a_n\}\subset C_0$
	and\\
	$\{b_1,\dots, b_n\}\subset C_1$, where $a_j = e^{2\pi j/n}$ and $b_j = e^{-2\pi j /n}$ with the orientation of $C_i$,
	\item $\beta$ is an $n$-tuple of arcs $\beta= (\beta_1, \dots, \beta_n)$ such that: 
	\begin{itemize}
		\item for $k\in\{1,\dots,n\}$, $\beta_k$ is an arc from $[0,1]$ to $S$, 
		\item for $k\in\{1,\dots,n\}$, $\beta_k(0)=a_k$ and there exists $\pi\in S_n$ such that $\beta_k(1) = b_{\pi(k)}$, 
		\item the set of $n$-tuple of curves $\beta$ is in general position, i.e. there are only transverse double points on the image of $\beta$ in $S$ (called {\it  crossings}), 
		\item the oriented graph formed by $\beta$ has no oriented cycles, 
	\end{itemize}
	\item each crossing is either a positive, negative or singular crossing with its nature indicated by a function  $$\epsilon: \text{\{Crossings\}} \rightarrow \{ +1 , -1, s\}.$$
\end{enumerate}
\end{definition}

\begin{definition}
We say that two singular abstract braid diagrams, $(S,\beta, \epsilon)$ and $(S,\beta', \epsilon')$, are {\it  Reidemeister equivalent}, if they are related by a finite number of the following operations: 
\begin{itemize}
    \item {\it Ambient isotopy.} There is a continuous map $$H:(S,\partial S) \times[0,1]\rightarrow (S, \partial S)$$ with $H_t=H(\cdot,t) \in Diff(S,\partial S)$ and $(S, H_t(\beta), \epsilon_t)$ is a singular abstract braid diagram such that $(S,H_0(\beta),\epsilon_0) = (S,\beta,\epsilon)$ and $(S,H_1(\beta),\epsilon_1) = (S,\beta',\epsilon')$, where $\epsilon_t$ is the crossing map induced by $H_t$. 
    \item {\it  Reidemeister moves.} We say that $(S,\beta, \epsilon)$ and $(S,\beta', \epsilon')$ are related by a Reidemeister move, if there exists an open neighbourhood in $S$ such that we can perform a Reidemeister move of type $R2$ or $R3$ on $\beta$ to obtain $(S,\beta', \epsilon')$
    \item {\it  Singular Reidemeister moves.} We say that $(S,\beta, \epsilon)$ and $(S,\beta', \epsilon')$ are related by a singular Reidemeister move, if there exists an open neighbourhood in $S$ such that we can perform a singular Reidemeister move of type $S3$ or $S4$ on $\beta$ to obtain $(S,\beta', \epsilon')$
\end{itemize} 
We call the set of Reidemeister equivalence classes the set of { \it singular abstract braids}, and we denote them by $SAB_n$. 
\end{definition}

\begin{definition}
We say that two singular abstract braid diagrams, $(S,\beta, \epsilon)$ and $(S',\beta', \epsilon')$, are {\it  Stable equivalent}, if they are related by a finite number of the following operations: 
\begin{enumerate}
	\item 
	{\it  Diffeomorphism.} We say that $(S',\beta',\epsilon')$ is obtained from $(S, \beta, \epsilon)$ by a diffeomorphism  if there exists $f \in Diff^{+}(S,S')$ such that $(S',\beta', \epsilon')=(f(S),f(\beta), \epsilon\circ f^{-1})$.
	

	\item {\it  Stabilization.} We say that $(S',\beta', \epsilon')$ is obtained from  $(S, \beta, \epsilon)$ by a stabilization if there exists an attaching region, $h:S^0\times D^2 \mapsto S$, for a $1$-handle that is disjoint from the image of $\beta$ and $(S',\beta', \epsilon')=(S'',\beta, \epsilon)$, where $S''$ is obtained by the $0$-surgery on $S$ along $h$, i.e. is the surface $$S'' := S \setminus \overbrace{h(S^0\times D^2)}^{\circ} \cup_{S^0\times S^1} D^1\times S^1.$$

	
	\item {\it  Destabilization.} A destabilization is the inverse operation of a stabilization, and it involves cutting $S$ along an essential curve $\gamma$ disjoint from the image of $\beta$ and attaching two copies of $D^2$ along the two new boundary components.  If the resulting surface is disconnected, then we keep only the component containing $\beta$. 
\end{enumerate}
\end{definition} 

We now prove that there exists a correspondence between singular virtual braids and singular abstract braids up to stable equivalence by establishing a bijection between these last objects and horizontal Gauss diagrams.

\begin{proposition}
	There is a bijective correspondence between the set of stable classes of singular abstract braids and the set of horizontal Gauss diagrams. 
\end{proposition}

\begin{proof}
	Notice that the Gauss diagram of a singular virtual braid diagram, $(\beta,\epsilon)$, is completely defined by the graph induced by $\beta$ and the function $\epsilon$. This allows us to define an analogue to Definition \ref{def:GaussBraidDiagram} for abstract braid diagrams, i.e. given an abstract braid diagram, $(S,\beta,\epsilon)$, we can associate to it a singular Gauss diagram $G(S,\beta,\epsilon)$.  
	
	
	Furthermore, this association is well defined up to stable equivalence, since diffeomorphisms, stabilizations and destabilizations do not change the pair  $(\beta,\epsilon)$, and if we perform a local Reidemeister or singular Reidemeister move, it is equivalent to perform an $\Omega$--move on the singular horizontal Gauss diagram.  Thus, we have a well defined function $$G: SAB_n / \text{(stability)}\rightarrow G_n.$$ 
	
	On the other hand, given a horizontal singular Gauss diagram, $g$, consider the singular virtual braid diagram $B(g)$.
\begin{figure}[h]
    \centering
    \includegraphics[scale=0.7]{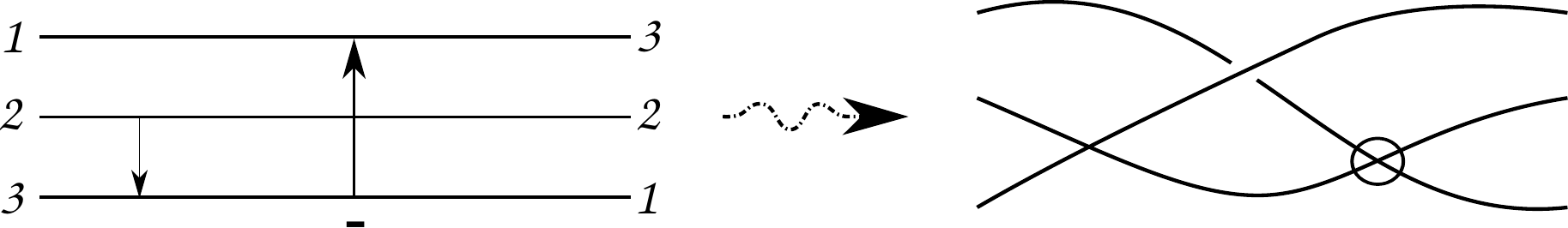} 
    \caption{Singular virtual braid diagram from Gauss diagram. }
    \label{fig:proof1}
\end{figure}

	 From $B(g)$ we can construct an abstract singular braid diagram as constructed in \cite{BACDLC} for braids and in \cite{Kamada-tachi} for knots, that is: 
\begin{enumerate}
	\item To each side of the braid diagram add a circle in such a way that the respective endpoints lie on it, call these circles {\it  distinguished components} (see Figure \ref{fig:proof2}).
\begin{figure}[h]
    \centering
    \includegraphics[scale=0.3]{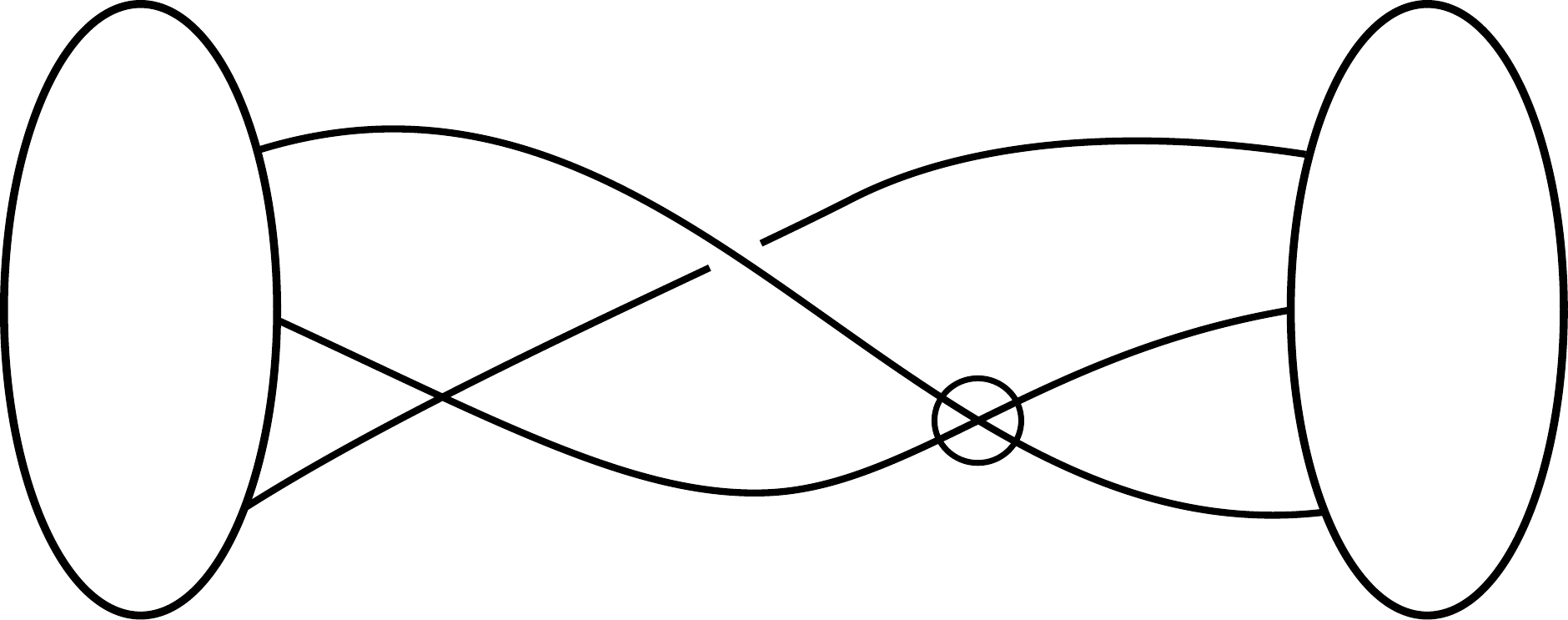} 
    \caption{Adding circles to the virtual braid diagram}
    \label{fig:proof2}
\end{figure}
	\item Take a regular neighborhood, in $\R^2$, of the obtained diagram. We get a surface $\Sigma$ with several boundary components, among them the distinguished components (see Figure \ref{fig:proof3}).
\begin{figure}[h]
    \centering
    \includegraphics[scale=0.3]{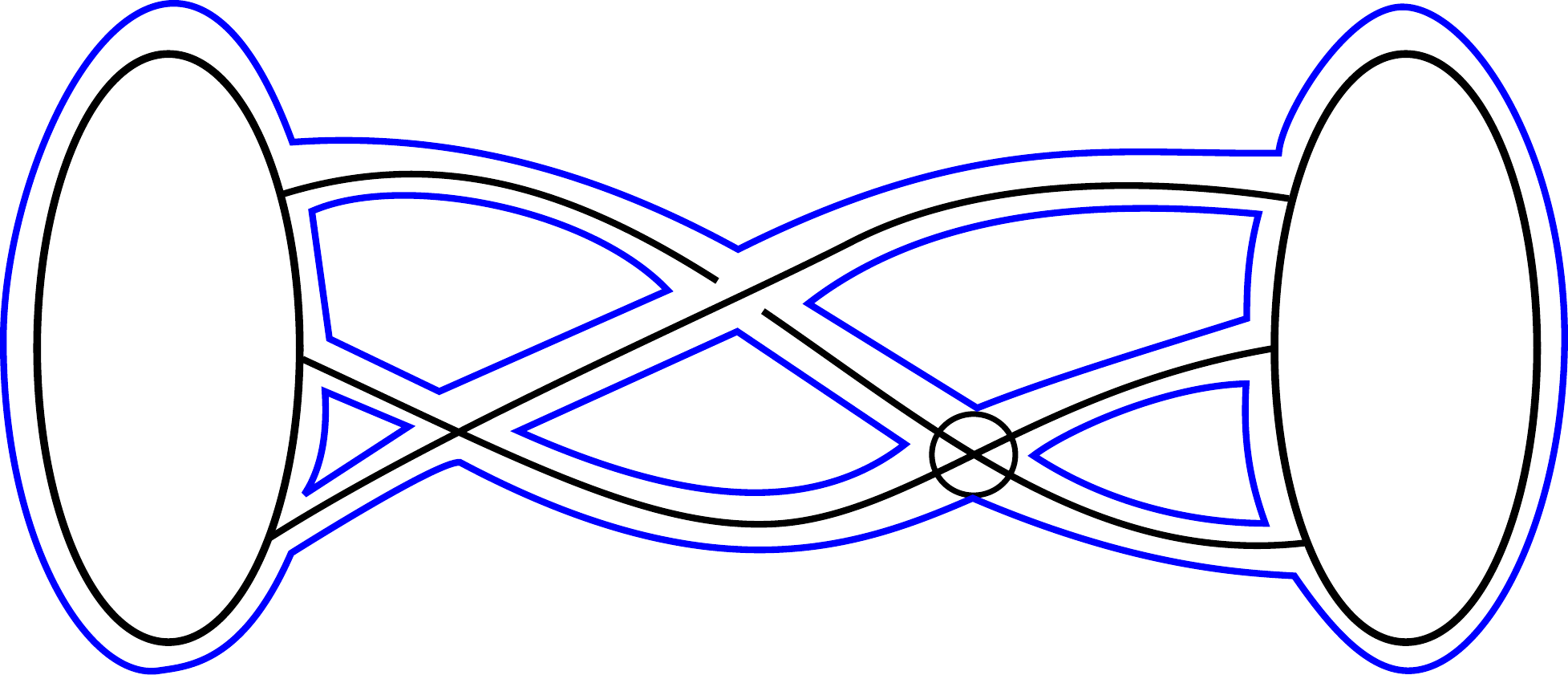} 
    \caption{The surface associated to the virtual braid diagram}
    \label{fig:proof3}
\end{figure}
	\item Consider the natural embedding of $\R^2$ in $\R^3$, this induces an embedding of the preceding surface. Perturb, in $\R^3$, a regular neighborhood of each virtual crossing in such a way that you obtain two disjoint bands. We obtain a new surface, preserving the distinguished boundary components (see Figure \ref{fig:proof4}). 
\begin{figure}[h]	
    \centering
    \includegraphics[scale=0.3]{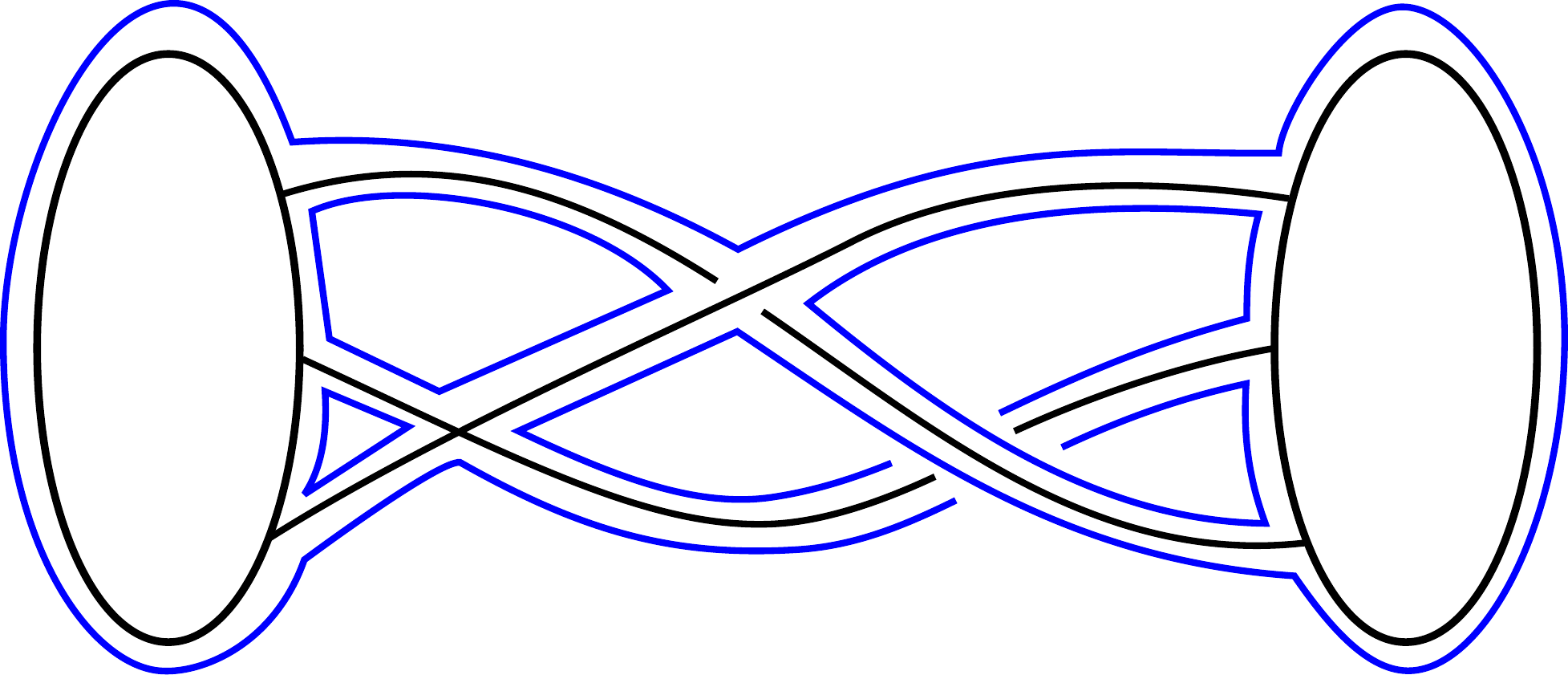} 
    \caption{The surface obtained after perturbations}
    \label{fig:proof4}
\end{figure}
	\item Consider the abstract surface obtained on the last construction and cap all the boundary components, but the distinguished boundary components. As the previous surface was oriented, we obtain an oriented surface with only two boundary components and satisfying the definition of abstract singular braid diagram. Call the abstract singular braid diagram $A(g) = (S,\beta,\epsilon)$ (see Figure \ref{fig:proof5}).
\begin{figure}[h]	
    \centering
    \includegraphics[scale=0.3]{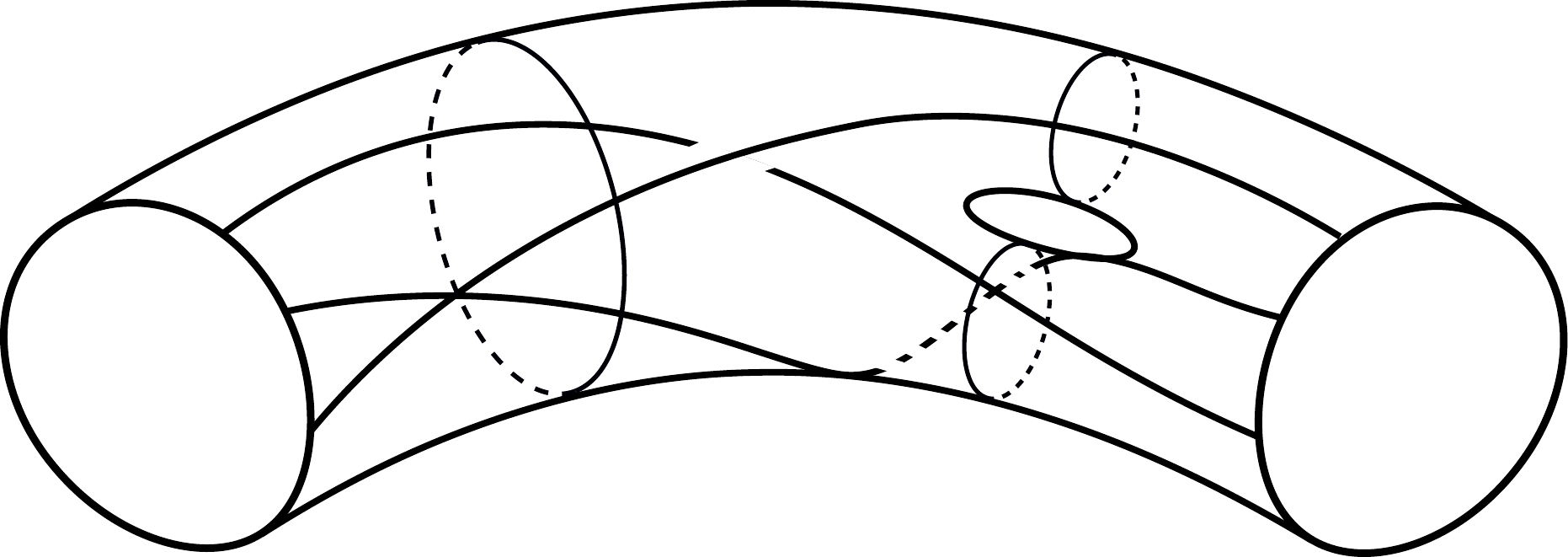} 
    \caption{The abstract braid diagram constructed from $\beta$.}
    \label{fig:proof5}
\end{figure}\end{enumerate}

Notice that for each virtual crossing we can perturb in two different ways, both are diffeomorphic, thus the surface obtained well-defined up to diffeomorphisms. Furthermore, if we choose a different singular braid diagram representing $g$, it only changes by virtual Reidemeister moves, thus when we perturb the regular neighbourhood we obtain the same surface up to diffeomorphism. From this, we have a well defined function from the set of singular horizontal Gauss diagram to the set of singular abstract braids. 

Finally, if we perform an $\Omega$--move on $g$, the associated abstract braid changes up to the correspondent Reidemeister move and (possibly) stabilization or destabilization. Thus, this defines a function $A: G_n \rightarrow SAB_n /\text{(stability)}$.

Note that $G \circ A = Id_{G_n}$ and $ A\circ G = Id_{SAB_n}$ and the bijection is given. 
\end{proof}

As an immediate consequence, we have the following proposition, which gives us a realization of singular virtual braids in a topological context :

\begin{proposition}\label{abst_bd}
	There is a bijection between the set of singular virtual braids and the set of stable classes of abstract singular braids. 
\end{proposition}

\bibliographystyle{plain}

\begin{thebibliography}{10}

\bibitem{ABMW}
Benjamin Audoux, Paolo Bellingeri, Jean-Baptiste Meilhan, and Emmanuel Wagner.
\newblock Homotopy classification of ribbon tubes and welded string links.
\newblock {\em Annali della scuola normale superiore di Pisa},
  XVII(2):713--761, 2016.

\bibitem{Bae}
John Baez.
\newblock Links invariants of finite type and perturbation theory.
\newblock {\em Letters in Mathematical Physics}, vol.26:43--51, 1992.

\bibitem{BARNATAN}
Dror Bar-Natan.
\newblock On the {V}assiliev knot invariants.
\newblock {\em Topology}, 34(2):423--472, 1995.

\bibitem{B}
Joan Birman.
\newblock New points of view in knot theory.
\newblock {\em Bulletin of the American Mathematical Society}, vol.28:253--287,
  1993.

\bibitem{CPM}
Carmen Caprau, de~la Pena, Andrew, and Sarah McGahan.
\newblock Virtual singular braids and links.
\newblock {\em Manuscripta Mathematica}, vol.151(1):147--175, 2016.

\bibitem{CZ}
Carmen Caprau and Sarah Zepeda.
\newblock On the virtual singulair braid monoid.
\newblock {\em arXiv:1710.05416v1}, 2017.

\bibitem{BACDLC}
Bruno Cisneros De La~Cruz.
\newblock Virtual braids from a topological point of view.
\newblock {\em Journal of Knot Theory and its Ramifications}, vol.24(6), 2015.

\bibitem{FRZ}
Roger Fenn, Dale Rolfsen, and Jun Zhu.
\newblock Centralizers in the braid group and singular braid monoid.
\newblock {\em L'Enseignement Math\'{e}matique}, vol.42:75--96, 1996.

\bibitem{GPV}
Mikhail Goussarov, Michael Polyak, and Oleg Viro.
\newblock Finite-type invariants of classical and virtual knots.
\newblock {\em Topology}, vol.39:1045--1068, 2000.

\bibitem{Kamada-tachi}
Naoko Kamada and Seiichi Kamada.
\newblock Abstract link diagrams and virtual knots.
\newblock {\em Journal of Knot Theory and its Ramifications}, vol.9(1):93--106,
  2000.

\bibitem{Kamada}
Seiichi Kamada.
\newblock Invariants of virtual braids and a remark on left stabilisations and
  virtual exchange moves.
\newblock {\em Kobe Journal of Mathematics}, vol.21:33--49, 2004.

\bibitem{Kauffman}
Louis Kauffman.
\newblock Virtual knot theory.
\newblock {\em European Journal of Combinatorics}, vol. 20(7):663--690, 1999.

\bibitem{PAPADIMA2002}
Stefan Papadima.
\newblock The universal finite-type invariant for braids, with integer
  coefficients.
\newblock {\em Topology and its Applications}, 118(1):169--185, 2002.
\newblock Arrangements in Boston: A Conference on Hyperplane Arrangements.

\bibitem{LP}
Luis Paris.
\newblock The proof of {B}irman{'}s conjecture on singular braid monoids.
\newblock {\em Geometry and Topology}, vol.8:1281--1300, 2004.

\bibitem{Ver01}
Vladimir Vershinin.
\newblock On homology of virtual braids and {B}urau representation.
\newblock {\em Journal of Knot Theory and Its Ramifications}, 10(05):795--812,
  2001.

\bibitem{Z}
Jun Zhu.
\newblock On singular braids.
\newblock {\em Journal of Knot Theory and its Ramifications}, vol.6:427--440,
  1997.

\end{thebibliography}

\end{document}